\documentclass[10pt]{amsart}

\usepackage{a4wide}
\usepackage{graphicx} 
\usepackage{geometry}
\usepackage[utf8]{inputenc}
\usepackage{amssymb,amsmath,amscd,amsfonts,amsthm,bbm,mathrsfs,enumerate}
\usepackage{booktabs}
\usepackage[colorlinks=true, linkcolor=blue, citecolor=blue]{hyperref}
\usepackage{mathabx}
\usepackage{graphicx}
\usepackage{color}
\usepackage{accents}
\usepackage{subcaption}
\usepackage{enumitem}
\usepackage{amsaddr}
\usepackage{comment}

\DeclareUnicodeCharacter{00A0}{~}

\DeclareMathOperator{\Var}{Var}

\newtheorem{problem}{Open Problem}

\newcommand{\cY}{{\check{Y}}}

\newcommand{\bs}{\boldsymbol}

\DeclareMathOperator{\bd}{bd}

\numberwithin{equation}{section}
\newcommand{\1}{\mathbbm 1}

\newcommand{\R}{\mathbb{R}}
\newcommand{\C}{\mathbb{C}}

\newcommand{\rank}{\mathop{\operator@font rank}}

\newcommand{\cI}{{\mathcal I}}
\newcommand{\cJ}{{\mathcal J}}

\newcommand{\vertiii}[1]{{\left\vert\kern-0.25ex\left\vert\kern-0.25ex\left\vert #1
    \right\vert\kern-0.25ex\right\vert\kern-0.25ex\right\vert}}
\makeatother

\DeclareMathOperator{\tr}{tr}
\DeclareMathOperator{\adj}{adj}

\DeclareMathOperator{\vect}{vec}

\newcommand{\HH}{\mathbb H}
\newcommand{\N}{\mathbb{N}} 
\newcommand{\E}{\mathbb{E}}

\newtheorem{notation}{Notation}

\newtheorem{thm}{Theorem}[section]
\newtheorem{lem}[thm]{Lemma}
\newtheorem{prop}[thm]{Proposition}
\newtheorem{cor}[thm]{Corollary}
\newtheorem{ass}{Assumption}
\theoremstyle{definition}

\theoremstyle{remark}
\newtheorem{rem}[thm]{Remark}

\title[Extreme eigenvalues and eigenvectors]{Extreme eigenvalues and 
eigenvectors for finite rank additive deformations of 
non-Hermitian sparse random matrices}
 
\author{Walid Hachem$^{a}$, Michail Louvaris$^{b}$, Jamal Najim$^{a}$}

\address{$^{a}$ CNRS, LIGM (UMR 8049), Université Gustave Eiffel, ESIEE Paris, France.
\quad \\ Emails: walid.hachem@univ-eiffel.fr, jamal.najim@univ-eiffel.fr}
\address{$^{b}$ Department of Mathematics, Yale University, New Haven, USA. \\
\quad Email: michail.louvaris@yale.edu}

\begin{document}
\maketitle
\begin{abstract}

Consider a $n\times n$ sparse non-Hermitian random matrix $X^n$ defined as the
Hadamard product between a random matrix with centered independent and
identically distributed entries and a sparse Bernoulli matrix with success
probability $K_n/n$ where $K_n \le n$ (and possibly $K_n\ll n$) and $K_n\to
\infty$ as $n\to \infty$. Let $E^n$ be a deterministic $n\times n$ finite-rank
matrix. We prove that the outlier eigenvalues of $Y^n= X^n +E^n$ asymptotically
match those of $E^n$. 

In the special case of a rank-one deformation, assuming further that the
sparsity parameter satisfies $K_n \gg \log^9 n$ and that the entries of the
random matrix are sub-Gaussian, we describe the limiting behavior of the
projection of the right eigenvector associated with the leading eigenvalue onto
the right eigenvector of the rank-one deformation. In particular, we prove
that the projection behaves as in the Hermitian case. To that end, we rely 
on the recent universality results of Brailovskaya and van Handel \cite{brailovskaya2024universality} relating
the singular value spectra of deformations of $X^n$ to Gaussian analogues of 
these matrices. 

Our analysis builds upon a recent framework introduced by Bordenave
\emph{et.al.} 2022 \cite{bordenave2022convergence}, and amounts to showing the asymptotic equivalence between
the reverse characteristic polynomial of the random matrix and a random
analytic function on the unit disc with explicit dependence on the finite-rank
deformation.

%
%

\end{abstract}

\section{Introduction and main results} 
The study of eigenvalue outliers in random matrix theory has a rich and
well-established history, particularly in the symmetric and Hermitian settings,
where additive finite-rank deformation often lead to predictable and
well-understood spectral deviations. A landmark result by Baik, Ben Arous, and
Péché (BBP) demonstrated that for sample covariance matrices with Gaussian entries, finite-rank
deformations can induce outlier eigenvalues that separate from the bulk
spectrum once a critical threshold is exceeded; see \cite{baik2005phase}. 
This so-called BBP transition was soon extended to general entries by Baik and Silverstein \cite{baik2006eigenvalues} and
has since become a foundational concept in the field,
with extensions to more general settings such as covariance-type matrices
\cite{paul2007asymptotics}, Wigner-type matrices \cite{capitaine2009largest},
and other deformed matrices \cite{benaych2011eigenvalues}. Key tools in these
developments include the resolvent method, master equations, and moments of
large power.

The non-symmetric / non-Hermitian setting introduces additional challenges;
nevertheless, significant progress has been achieved. In particular,
\cite{tao2013outliers} and \cite{bordenave2016outlier} provide a complete
characterization of the outlier distribution in the i.i.d.~case, assuming
finite fourth moments for the entries.

More recently, the sparse circular law has been established under minimal
moment assumptions in \cite{rudelson2019sparse} and \cite{sah2025sparse}.
Building upon these advances, we prove that outlier results continue to hold
across all sparsity regimes. Our main technical tool is the analysis of the
reverse characteristic polynomial, as developed in
\cite{bordenave2022convergence}. Furthermore, under additional assumptions on
the matrix and its sparsity parameter, we establish a result concerning the
right-eigenvector associated with the largest eigenvalue in the case where the
additive deformation has rank one. To this end, we compare spectral quantities
of the matrix with those of an analogous Gaussian ensemble, leveraging
universality results from \cite{brailovskaya2024universality}.

By also relying on the technique of~\cite{bordenave2022convergence}, the author
of the recent paper \cite{han-25} also deals with the outliers induced by
finite rank deformations of square matrices with independent and identically
distributed entries. This paper deals among others with the sparse Bernoulli
case with a finite rank additive deformation, a model close to ours. The
sparsity parameter of the Bernoulli elements is assumed to converge to infinity
at the rate $n^{o(1)}$.  In this situation, it is moreover assumed in
\cite{han-25} that the finite rank deformation has a finite number of non-zero
elements. These assumptions are not required in our paper, where we only need
the deformation to have a bounded operator norm. Moreover, we do not put any
assumption on the rate of increase of the sparsity parameter. In addition, when
our deformation is of rank one, we also study the angle between the eigenvector
associated to the outlier and the ``true'' vector, a problem not considered
in~\cite{han-25}.  On the other hand, \cite{han-25} tackles the problem of the
extreme eigenvalues of finite product of matrices.

Random additively deformed non-Hermitian matrices
appear in many applied fields, such as natural and artificial neural networks where the
random matrix $Y^n$ at hand represents the random interactions between the
neurons~\cite{som-cri-som-88,wai-tou-13}. We may also cite theoretical ecology where $Y^n$, which
is often sparse, models the interactions among living species within an 
ecosystem~\cite{bunin2017ecological,akj-etal-24}, see also the references therein. In these fields, the eigenvalue of $Y^n$ 
with the largest modulus plays a central role in describing the time evolution 
of the activity of $n$ interacting neurons or of the abundances of the $n$ species 
that constitute the ecosystem.

We introduce some notation before stating our results. 

\subsection{Notations} Let $\C^+ = \{ z \in \C \, : \, \Im z > 0 \}$. The
cardinality of a set $S$, counting multiplicities, is denoted by $| S |$.  For
$m\in \mathbb{N}$, set $[m] = \emptyset$ if $m = 0$ and $[m] = \{1,\ldots, m
\}$ otherwise. Let $z\in \mathbb{C}$ and $A,B\subset \mathbb{C}$, then $d(z,A)=\inf_{\xi\in A}|z-\xi|$ and the Hausdorff distance between $A$ abd $B$, denoted by $d_{\bs{H}}(A,B)$ is defined by
$$
d_{\bs{H}}(A,B) =\max\left\{ \sup_{z\in A}d(z,B)\,;\ \sup_{z\in B}d(z,A) \right\}\, .
$$
When $m >0$, we denote as $\mathfrak S_m$ the symmetric group
over the set $[m]$.  Let $\| \cdot\|$ be the matrix operator norm or the vector
Euclidean norm.  For a matrix $M$, denote by $M^\star$ its conjugate transpose; if
$u,v$ are column vectors with equal dimension, then $\langle u, v\rangle =u^\star
v$. Denote by $I_m$ the $m\times m$ identity matrix, or simply $I$ if the
dimension can be inferred from the context. Denote by $\sigma(M) = \{
\lambda_1(M), \ldots, \lambda_m(M) \}$ the spectrum of a $m\times m$ matrix
$M$, by $\rho(M)$ its spectral radius, and by $s_m(M)$ its least singular
value. For a $m \times m$ matrix $M = (M_{ij})_{i,j=1}^m$ and $\cI,\cJ\subset
[m]$, let $M_{\cI,\cJ}=(M_{ij})_{i\in \cI, j\in \cJ}$ and
$M_\cI=(M_{ij})_{i,j\in \cI}$.  Denote by $\adj(M)$ the adjugate of $M$,
\emph{i.e.}, the transpose of $M$'s cofactors matrix.  For a vector $x\in
\mathbb{C}^m$ and ${\mathcal I}\subset [m]$ let $x_{\mathcal I}=(x_i)_{i\in
{\mathcal I}}$.

For a sequence of random variables $(U_n)$ and a random variable $U$ with values in a common metric space, denote by  $U_n\underset{n \to \infty}{\overset{\mathbb{P}}{\longrightarrow}}U$ and $
U_n\underset{n \to \infty}{\overset{\text{law}}{\longrightarrow}}U$ the convergence in probability and in law, respectively. 
Let $U_n$ and
$V_n$ be random variables in some metric space with probability distribution $\mu_n$ and $\nu_n$. The notation 
\[
  U_n \sim V_n\quad (n\to\infty)
\]
refers to the fact that the sequences $(\mu_n)$ and $(\nu_n)$ are relatively
compact, and that 
\[
   \int f d\mu_n - \int f d \nu_n \xrightarrow[n\to\infty]{}  0 \qquad \left( \Leftrightarrow \quad \mathbb{E} f(U_n) - \mathbb{E}f(V_n) \xrightarrow[n\to\infty]{}  0 \right)
\]
for each bounded continuous real function $f$ on the metric space. We shall say
then that $(U_n)$ and $(V_n)$ are ``asymptotically equivalent''. Note that
$(\mu_n)$ and $(\nu_n)$ do not necessarily converge narrowly to some
probability distribution. We denote by $\nu_n\Rightarrow_n \nu$ the weak convergence of probability measures.

Let $f:A\subset {\mathcal X} \to \R$. We define the function $1_A f$ by
$$
1_A(x) f(x) = \begin{cases}
    f(x)&\textrm{if}\ x\in A\,,\\
    0&\textrm{else.}
\end{cases}
$$

Denote by $\mathcal D(a,\rho)$ the open disk of $\C$ with center $a \in
\C$ and radius $\rho > 0$, by $\mathbb{H}$ the space of
holomorphic functions on $\mathcal{D}(0,1)$, equipped with the topology of
uniform convergence on compact subsets of $\mathcal{D}(0,1)$. It is well-known that $\mathbb{H}$ is a polish space. 

The following conventions will be used throughout the article: $\sum_{\emptyset}=0$, $\prod_{\emptyset}=1$, $\mathrm{det}(A)=1$ if $A$ is a matrix of null dimension. For complex sequences $(w_n),(\tilde w_n)$, the notation $u_n={\mathcal O}(v_n)$ implies the existence of a positive constant $\kappa$  such that $|u_n|\le \kappa |v_n|$ for all $n\ge 1$ sufficiently large. If we want to emphasize the fact that the constant $\kappa$ depends on some extra parameters $z,\eta$, we may write $u_n={\mathcal O}_{z,\eta}(v_n)$.

\subsection{Main results}

\subsubsection{The model}

We begin by introducing our random matrix model.  Let $\chi$ be a complex-valued
random variable such that $\mathbb{E}(\chi) = 0$ and $\mathbb{E}(|\chi|^2) = 1$.
For each integer $n \ge 1$, let 
$A^n  = ( A_{ij}^n )_{i,j=1}^n \in \C^{n\times n}$ be a random matrix with
independent and identically distributed (i.i.d.) elements equal in distribution
to $\chi$. 

Let $(K_n)$ be a sequence of positive integers such that $K_n \leq n$. Let  
$(B^n)$ be a sequence of $n \times n$ matrices with i.i.d. Bernoulli entries
such that, writing $B^n = ( B^n_{ij} )_{i,j=1}^n$, we have $\mathbb P\{ B^n_{11}
= 1\,\} = K_n / n$.  We also assume that $B^n$ and $A^n$ are independent. We
consider the sequence of $n\times n$ random matrices $(X^n)_{n\ge 1}$ given 
as follows.  Writing $X^n = ( X^n_{ij} )_{i,j=1}^n$, we set 
\begin{align}
\label{sparser_matrices} 
X^n_{ij} = \frac{1}{\sqrt{K_n}} B^n_{ij}\, A^n_{ij}\ .  
\end{align} 
Notice that $\mathbb{E} X_{11}^n=0$ and $\mathbb{E} |X_{11}^n|^2= 1/n$. 

Let $r > 0$ be a fixed integer, and consider $2r$ sequences of deterministic
vectors $(u^{1,n})$, $(u^{2,n})$, ..., $(u^{r,n})$, $(v^{1,n})$, $(v^{2,n})$, 
 ..., $(v^{r,n})$ such that $u^{t,n}, v^{t,n} \in \C^n$ for each $t\in[r]$ and each 
$n > 0$.  Consider the sequence $(E^n)$ of $n\times n$ deterministic matrices
defined by 
\[
 E^n = \sum_{t=1}^r u^{t,n} (v^{t,n})^\star\,.
\] 
We make the following assumptions: 
\begin{ass}
\label{ass:K} The integer sequence $(K_n)$ satisfies
    $$K_n \xrightarrow[n\to\infty]{}\infty\, .
    $$
\end{ass}
\begin{ass}
\label{ass:E} 
    There exists an absolute constant $C > 0$ such that
\[
   \sum_{t=1}^r \|u^{t,n}\| + \|v^{t,n} \|  \leq C. 
\]
\end{ass}
In many applicative contexts, $(K_n)$ converges to infinity at a much slower
pace than $n$. For this reason, the parameter $K_n$ is referred to as the
sparsity parameter of the model of $X^n$.

Define the sequence of random matrices $(Y^n)$ as 
\[
    Y^n = X^n + E^n.
\] 
It is well-known, see \cite[Theorem 1.4]{sah2025sparse} which generalizes
\cite[Theorem 1.2]{rudelson2019sparse}, that the empirical spectral
distribution of $X^n$ converges to the so-called circular law. We shall
furthermore show in Theorem \ref{rho-sparse} below that the
spectral radius of $X^n$ converges to $1$.  In this article, we study the
asymptotic behavior of the eigenvalues of $Y^n$ which Euclidean norm is greater
than $1$.  We refer to these eigenvalues as \textit{outliers}, which presence
is due to $E^n$. Their behavior will be described in
Theorem~\ref{thm:largest_eigenvalue} below.  In the case of a single outlier,
we describe the behavior of the associated eigenvector. This will be the 
content of Theorem~\ref{thm_eigenvector}.

\subsubsection{Eigenvalues and characteristic polynomial of $Y^n$}
Our approach is inspired by the technique developed in
\cite{bordenave2022convergence} to capture the asymptotic behavior of the
spectral radius of random matrices with i.i.d.~elements, and later extended in
\cite{coste2023sparse}, \cite{cos-lam-yiz-24}, \cite{fra-gar-23}, and
\cite{hachem2025spectral} to various other models. One key feature of this
approach is that it requires minimal assumptions on the moments of the random
matrices' entries, and it is based on analyzing the asymptotic behavior of the
reverse characteristic polynomial via convergence to a random analytic function
in the unit disk. This latter idea can be found in~\cite{shirai2012limit}.

Consider the reverse characteristic polynomial of matrix $Y^{n}$, defined by
\begin{equation}\label{def:q}
q_n(z) = \det\left(I_n - z Y^{n}\right)\,.
\end{equation}
Clearly, $q_n$ is a $\mathbb{H}$-valued random variable. In this paper, our
first goal is to study the asymptotic distribution of $q_n$ on $\mathbb{H}$.
More precisely, we seek an appropriate sequence of random analytic functions
$\varphi_n\in \mathbb{H}$ such that
\[
q_n \sim \varphi_n\,,\quad (n\to\infty)\, ,
\]
where $\varphi_n$ is simpler to analyze that $q_n$. Studying the large-$n$ behavior 
of $q_n$ in the light of the notion of asymptotic equivalence is well-suited 
to our purpose, since without additional assumptions on the matrices $E^n$, there is
no reason for $(q_n)$ to converge in law in $\mathbb{H}$.

In what follows, we define the sequence of polynomials $(b_n)$ as 
$$
b_n(z) =\det(I-zE^n)\,.
$$
This sequence is pre-compact in $\HH$ as a sequence
of polynomials with degrees bounded by $r$ and with bounded coefficients by
Assumption~\ref{ass:E}. 

\begin{thm}\label{characteristic_poly_thm}
Let Assumptions \ref{ass:K} and \ref{ass:E} hold true. Consider a sequence 
$(Z_k)_{k \geq 1}$ of independent Gaussian random variables with
$$
 \mathbb{E}(Z_k) = 0, \quad \mathbb{E}(|Z_k|^2) = 1, 
  \qquad \text{and} \qquad \mathbb{E}(Z_k^2) = (\mathbb{E}A_{11}^2)^k\,.
$$
Define
\[
\kappa(z) = \sqrt{1 - z^2 \mathbb{E}A_{11}^2} 
 \quad \text{with} \quad \sqrt{1} = 1, \quad \text{and} \quad 
F(z) = \sum_{k=1}^\infty z^k \frac{Z_k}{\sqrt{k}} 
 \quad \text{for} \ z \in \mathcal D(0,1). 
\] 
Also let $G_n(z)= b_n(z) \det(I - z X^n)$. Then
\begin{align}\label{thm_statement_char_pol}
    q_n \sim G_n\,,\quad (n\to\infty)
\end{align}
as $\mathbb H$--valued random variables. Also, 
\begin{align}\label{q_nsimthm}
    q_n \sim b_n \, \kappa \, \exp(-F)\,,\quad (n\to\infty)\, ,
\end{align}
as $\mathbb H$--valued random variables.
\end{thm}

Proof of Theorem \ref{characteristic_poly_thm} is given in Section
\ref{section_char_poly}. 

This theorem captures the behavior of the eigenvalues of $Y^n$ which are away
from the unit-disk. In a word, since $\det(I - z Y^n) \sim \det(I-z E^n)
\kappa(z) \, \exp(-F(z))$ and since the function $z\mapsto \kappa(z)
\exp(-F(z))$ has no zero in $\mathcal D(0,1)$, these eigenvalues are close for
large $n$ to their counterparts for $E^n$.  This is formalized in the next 
theorem which generalizes Theorem 1.7 of \cite{tao2013outliers} to sparser
regimes. We need the following assumption. 

\begin{ass}\label{ass:eigenseparation}
There exists $\varepsilon>0$ such that
$
\sigma(E^n) \,\cap\, 
 \{ z \in \C \, : \, 1 < |z| < 1 + \varepsilon \} = \emptyset
$
for all large $n$. 
\end{ass}

\begin{thm}
\label{thm:largest_eigenvalue} 

Let Assumptions \ref{ass:K} and \ref{ass:E} hold. Assume that
Assumption~\ref{ass:eigenseparation} holds for some $\varepsilon>0$. 
Define the set
$$
\sigma^+(E^n) = \sigma(E^n) \cap \{ z \in \C \, : \, |z| > 1
\}\qquad \textrm{and}\qquad \sigma^+_\varepsilon(Y^n) = \sigma(Y^n) \cap \{ z \in \C \, : \, |z| \geq 1 +
\varepsilon \}
$$ and let $m_n = | \sigma^+(E^n) |$. Then, 
\[
\mathbb P\left\{ | \sigma^+_\varepsilon(Y^n) | \neq m_n \right\}
 \xrightarrow[n\to\infty]{} 0 \ .
\]
For each sequence $(n')$ converging to infinity such that 
$m_{n'} > 0$ for each $n'$, the Hausdorff distance between the sets $\sigma^+_\varepsilon(Y^{n'})$ and $\sigma^+(E^{n'})$ satisfies:
$$d_{\bs H}(\sigma^+_\varepsilon(Y^{n'}), \sigma^+(E^{n'})) \xrightarrow[n\to\infty]{\mathbb{P}} 0\, 
$$ 
(here, we set 
$d_{\bs H}(\emptyset, \sigma^+(E^{n'})) = \infty$). 
\end{thm}
\begin{proof}[Proof of Theorem \ref{thm:largest_eigenvalue} given Theorem
\ref{characteristic_poly_thm}]

To prove the first assertion, assume towards a contradiction that there exists
a sequence $(\tilde n)$ converging to infinity such that $\liminf_n \mathbb
P\left\{ | \sigma^+_\varepsilon(Y^{\tilde n}) | \neq m_{\tilde n} \right\} > 0$.
From this sequence, extract a subsequence also denoted as $(\tilde n)$ such
that $b_{\tilde n}$ converges to some $b_\infty$ in $\HH$. Notice that
$b_\infty$ is a polynomial with a degree bounded by $r$.  By
Assumption~\ref{ass:eigenseparation}, $b_\infty$ has no zero in the ring
$(1+\varepsilon)^{-1} < |z| < 1$.  Let $m_\infty\leq r$ be the number of zeros
of $b_\infty$ in $\mathcal D(0,1)$. When $m_\infty > 0$, let $\{ \zeta_1,
\cdots, \zeta_{s_\infty} \}$ be the set of these zeros not counting
multiplicities, where $s_\infty\leq m_\infty$ is the number of these zeros.  In
this case, denote as $k_i$ the multiplictity of the zero $\zeta_i$ for
$i\in[s_\infty]$, and define the set $\Lambda_\infty = \{  1/\zeta_1, \cdots,
1/\zeta_{s_\infty} \}$.  Then, it holds by, \emph{e.g.}, Rouché's theorem that
$m_{\tilde n} = m_\infty$ for all large $\tilde n$, and furthermore, if
$m_\infty > 0$, that the Hausdorff distance $d_{\bs H}(\sigma^+(E^{\tilde n}),
\Lambda_\infty)$ converges to zero. Indeed, by this theorem, there are $k_1$
eigenvalues of $E^{\tilde n}$ that converge to $1/\zeta_1$, ..., $k_{s_\infty}$
eigenvalues of $E^{\tilde n}$ that converge to $1/\zeta_{s_\infty}$, and these
eigenvalues exhaust $\sigma^+(E^{\tilde n})$ for all large $\tilde n$.

We shall show that 
\begin{equation}
\label{cvg-sigmaY}
| \sigma^+_\varepsilon(Y^{\tilde n}) | \quad \xrightarrow[n\to\infty]{\mathbb{P}}\quad  m_\infty \,, 
\end{equation}
obtaining our contradiction. 

By Theorem~\ref{characteristic_poly_thm}, $q_{\tilde n}$ converges in law
towards the $\mathbb H$--valued random function $q_\infty(z) = b_\infty(z)
\kappa(z) \exp(-F(z))$.  By relying on the explicit expressions of $\kappa$ and
$F$, notice that function $\kappa \exp( -F)$ does not vanish on ${\mathcal
D}(0,1)$.  If $m_\infty = 0$, then $q_\infty$ does not vanish on ${\mathcal
D}(0,1)$ either.  Otherwise, the set of zeros of $q_\infty$ coincides with $\{
\zeta_1, \cdots, \zeta_{s_\infty} \}$ with the same multiplicities. 

By Skorokhod's representation theorem, there exists a sequence of $\mathbb
H$--valued random variables $(\check q_{\tilde n})$ and a $\mathbb H$--valued
random variable $\check q_\infty$ defined on some common probability space
$\widecheck{\Omega}$, such that $\check q_{\tilde n} \stackrel{\text{law}}{=}
q_{\tilde n}$, $\check q_\infty \stackrel{\text{law}}{=} q_\infty$, and 
$\check q_{\tilde n}$ converges to $\check q_\infty$ for all 
$\check \omega \in \widecheck{\Omega}$. 

We now fix $\check\omega$ and apply Rouché's theorem.  If $m_\infty = 0$, then
$\check q_{\tilde n}$ has eventually no zero in the compact set $\{ z \, : \,
|z|\leq 1/(1+\varepsilon) \}$.  Otherwise, $\check q_{\tilde n}$ has $k_1$
zeros converging to $\zeta_1$, $\cdots$, $k_{s_\infty}$ zeros converging to
$\zeta_{s_\infty}$, and these zeros exhaust the zeros of $\check q_{\tilde n}$
in $\{ z \, : \, |z|\leq 1/(1+\varepsilon) \}$ for all large $\tilde n$. 

Getting back to $q_{\tilde n}$, it remains to notice that the zeros of 
$q_{\tilde n}$ in $\{ z \, : \, |z|\leq 1/(1+\varepsilon) \}$, when they exist,
are the inverses of the eigenvalues of $Y^{\tilde n}$ in the set 
$\{ z \, : \, |z| \geq 1+\varepsilon \}$. This establishes the 
convergence~\eqref{cvg-sigmaY}. 

The proof of the second assertion of Theorem~\ref{thm:largest_eigenvalue} 
follows the same canvas. We just exclude the case where $m_\infty = 0$. 
\end{proof} 

Taking $E^n=0$, we obtain the following result.
\begin{cor}\label{cor_spectral_radius} Let Assumption \ref{ass:K} hold and let $\rho(X^n)$ be the spectral radius of $X^n$, then for every $\varepsilon>0$, we have
$
\mathbb{P}(\rho(X^n)>1+\varepsilon) \xrightarrow[n\to\infty]{} 0\, .
$
\end{cor}
Combining this corollary with the circular law for sparse matrices \cite[Theorem 1.4]{sah2025sparse}, we can generalize \cite[Theorem 1.1]{bordenave2022convergence} to the sparse case and get:
\begin{thm}
\label{rho-sparse} 
Let Assumption \ref{ass:K} hold and let $\rho(X^n)$ be the spectral radius of $X^n$, then 
$$
\rho(X^n)\xrightarrow[n\to\infty]{\mathbb{P}} 1\, .
$$    
\end{thm}

\subsubsection{Eigenvectors of rank-one deformation}
We now restrict our attention to rank-one deformations. Assuming that $r =1$,
write $u^n = u^{1,n}$ and $v^n = v^{1,n}$ for simplicity. The deformation 
matrix becomes then $E^n = u^n (v^n)^\star$. We need the following assumption: 
\begin{ass}\label{ass:vectors} 
The deterministic sequences $(u^n)$ and $(v^n)$ satisfy:  
\[
\liminf_{n \to \infty} \left| \langle v^n,u^n\rangle \right|  \ >\  1\,. 
\] 
\end{ass}
Obviously, $E^n$ is a square $n\times n$ matrix which only non-zero eigenvalue
is $\langle v^n,u^n\rangle$. By the previous assumption, $(E^n)$ satisfies 
Assumption~\ref{ass:eigenseparation}, and we imediately have the following 
result: 
\begin{cor}[corollary to Theorem~\ref{thm:largest_eigenvalue}] 
\label{lmax-uv}
Let Assumptions \ref{ass:K}, \ref{ass:E}, and \ref{ass:vectors} hold true.  
For any fixed 
$\varepsilon \in (0, (\liminf | \langle v^n,u^n\rangle | - 1) / 2)$, consider
the set $\sigma^+_\varepsilon(Y^n)$ defined in the statement of 
Theorem~\ref{thm:largest_eigenvalue}. Then, 
\[
\mathbb{P}\left\{ | \sigma^+_\varepsilon(Y^n) | = 1 \right\} \ 
   \xrightarrow[n\to\infty]{}\ 1\,.
\]
When the event $\left[ | \sigma^+_\varepsilon(Y^n) | = 1 \right]$ is
realized, let $\lambda_{\max}(Y^n) $ be the unique eigenvalue of $Y^n$ with 
the largest modulus, otherwise set $\lambda_{\max}(Y^n) = 0$. Then, 
\[
\lambda_{\max}(Y^n) - \langle v^n,u^n\rangle \
   \xrightarrow[n\to\infty]{\mathbb{P}}\  0\,.
\] 
\end{cor} 
In the remainder, when we mention the event $\left[ | \sigma^+_\varepsilon(Y^n) | = 1 \right]$,
we assume that $\varepsilon >
0$ is small enough according to the statement of the previous corollary.  Our
objective is to analyze the projection on $u^n$ of the right eigenvector of
$Y^n$ corresponding to $\lambda_{\max}(Y^n)$ (assuming $\left[ | \sigma^+_\varepsilon(Y^n) | = 1 \right]$ is
realized).  Our main technical tools are based on the results from
\cite{brailovskaya2024universality}, which allow us to compare the spectral
properties of $X^n$ with a Gaussian analogue to this matrix.  We will need the
following extra sub-Gaussian assumption concerning $A^n$'s entries.
\begin{ass}[sub-Gaussiannity]\label{ass:subgaussian}
The random variables $A_{ij}$ follow a sub-Gaussian distribution, \emph{i.e.},
there exists an absolute constant $C > 0$ such that 
\[
 \mathbb{P}(|A^n_{11}| \geq t) \ \leq\  2 \exp(-C t^2)\,. 
\] 
\end{ass}

We are now in position to describe the eigenvectors of $Y^n=X^n+u^n(v^n)^\star$
corresponding to the outlier $\lambda_{\max}(Y^n)$. 
\begin{thm}\label{thm_eigenvector}
Let Assumptions \ref{ass:K}, \ref{ass:E}, \ref{ass:vectors} and 
\ref{ass:subgaussian} hold true. Assume furthermore that 
\[
  \lim_{n \to \infty}\frac{\log^9 n}{K_n}\ =\ 0\,.
\] 
When the event $\left\{ | \sigma^+_\varepsilon(Y^n) | = 1 \right\}$ is realized, 
let $\tilde{u}^n$ be an 
unit-norm right eigenvector of $Y^n$ corresponding to $\lambda_{\max}(Y^n)$.  
Otherwise, put $\tilde{u}^n = 0_n$. Then, it holds that 
\[
 \left|\left\langle\tilde{u}^n,\frac{u^n}{\|u^n\|} \right\rangle\right|^2 - \left( 1-\frac 1{|\langle u^n,v^n\rangle|^2}\right)\quad \xrightarrow[n\to\infty]{\mathbb{P}}\quad 0\, .
\]
\end{thm}
Proof of Theorem \ref{thm_eigenvector} is postponed to Section \ref{section_eigenvectors}. 

\begin{rem}
In the case where $u^n$ is a unit-norm vector and where one considers the model
$Y^n= X^n+ \alpha u^n(u^n)^\star$  for some fixed $\alpha>1$, then the result above
boils down to $$
\left|\left\langle\tilde{u}^n,u^n \right\rangle\right|^2
\ \xrightarrow[n\to\infty]{\mathbb{P}} \ 1- \frac{1}{\alpha^2}\, .
$$
Interestingly, this corresponds
to the same quantity as in the Hermitian case, 
see \cite[Section 3.1]{benaych2011eigenvalues}.  
\end{rem}

\section{Proof of Theorem \ref{characteristic_poly_thm}}
\label{section_char_poly}

This section is devoted to the proof of Theorem \ref{characteristic_poly_thm}. We follow the strategy developed in \cite{bordenave2022convergence}.

\subsection{Tightness and truncation}

We first state useful properties for $\mathbb{H}$-valued random variables.

\begin{prop}[Tightness criterion {\cite[Proposition
3.1]{hachem2025spectral}}]
\label{prop:tightness}
    Let $(f_n)$ be a sequence of $\mathbb{H}$-valued random variables. If for every compact set $K\subset {\mathcal D}(0,1)$, 
    $$
    \sup_n \sup_{z\in K} \mathbb{E}| f_n(z)|^2 \le C_K <\infty\, ,
    $$
    for some $K$-dependent constant $C_K$, then $(f_n)$ is tight.
\end{prop}

\begin{prop}[Asymptotic equivalence criteria in $\mathbb{H}$]
\label{cvgH} 
    Let $(f_n)$ and $(g_n)$ be two tight sequences of $\mathbb{H}$-valued random variables. Consider their power series representations in ${\mathcal D}(0,1)$: $f_n(z) =\sum_{k=0}^\infty a_k^{(n)}z^k$ and $g_n(z)=\sum_{k\ge 0} b_k^{(n)} z^k$. If one of the following conditions holds:
    \begin{enumerate}
        \item For every fixed integer $m\ge 1$, $(a_0^{(n)},\cdots, a_m^{(n)})\ \sim_n \ (b_0^{(n)},\cdots, b_m^{(n)})$, 
        \item For every fixed integer $m\ge 1$ and $m$-uple $(z_1,\cdots z_m)\in {\mathcal D}^m(0,1)$, 
        $$
        \left( f_n(z_1),\cdots, f_n(z_m)\right) \ \sim_n \  \left( g_n(z_1),\cdots, g_n(z_m)\right)\ ,
        $$
    \end{enumerate}
    then $f_n \ \sim g_n$ as $n\to\infty$.
\end{prop}

Most of the time, we shall drop the dependence in $n$ for notational
convenience.

\begin{prop}[Tightness]
\label{tightness-prop}
Let Assumptions \ref{ass:K} and \ref{ass:E} hold. Let $q_n$ be given by \eqref{def:q}, then the sequence 
$(q_n)_{n \geq 1}$ is tight in $\mathbb{H}$.  
\end{prop}
\begin{proof} 
We first recall a well-known general result. Let $A$ and $B$ be two $n\times n$ 
matrices with columns $A_{i}$ and $B_i$ respectively for $i\in[n]$. Then, 
using the multilinearity of the determinant, we can write 
\begin{align*} 
\det(A+B) &= \det\begin{bmatrix} A_1+B_1 & A_2+B_2 & \cdots & A_n + B_n 
    \end{bmatrix}  \\ 
 &= \det\begin{bmatrix} A_1 & A_2+B_2 & \cdots & A_n + B_n \end{bmatrix}   
  + \det\begin{bmatrix} B_1 & A_2+B_2 & \cdots & A_n + B_n \end{bmatrix}  \\
 &= \det\begin{bmatrix} A_1 & A_2 & \cdots & A_n + B_n \end{bmatrix}   
 + \det\begin{bmatrix} A_1 & B_2 & \cdots & A_n + B_n \end{bmatrix}  + \cdots
\end{align*} 
which will ultimately provide a ``binomial-like'' expression of 
$\det(A+B)$ that will have the following form. Given $k \in \{0,\ldots, n\}$,
let $\cI \in [n]$ with $|\cI| = k$ and all the elements of $\cI$ are different, 
and denote as $(A,B)_{\cI}$ the $n\times n$ matrix which $i^{\text{th}}$ 
column is $A_i$ if $i\in \cI$ and $B_i$ if $i\in [n] \setminus \cI$. 
Then, 
\begin{equation}
\label{detA+B}
\det(A+B) = \sum_{k=0}^n \sum_{\cI\in[n] : |\cI| = k} \det(A,B)_\cI .
\end{equation}
Let us write $M = I - z E = [M_{ij}]_{i,j=1}^n$, so that 
$q(z) = \det(- z X + M)$. Writing 
\[
\E |q(z)|^2 = \sum_{\sigma,\tilde\sigma \in \mathfrak S_n} 
 \E ( - z X_{1,\sigma(1)} + M_{1,\sigma(1)}) \ldots 
  ( - z X_{n,\sigma(n)} + M_{n,\sigma(n)} ) 
 ( - \bar z \bar X_{1,\tilde\sigma(1)} + \bar M_{1,\tilde\sigma(1)}) \ldots 
  ( - \bar z \bar X_{n,\tilde\sigma(n)} + \bar M_{n,\tilde\sigma(n)}) , 
\]
we see that the element $X_{ij}$ acts on $\E |q(z)|^2$ through 
$\E X_{ij}$ and $\E | X_{ij}|^2$ only. Therefore, $\E |q(z)|^2$ is invariant
if we assume that these elements are i.i.d.~with 
$X_{11} \sim \mathcal N_\C(0,1/n)$, which we do from now on in this proof.

Denoting as $M = U \Sigma V^\star$ a singular value decomposition of $M$, we have
\[
|q(z)|^2 = \det(-zX + M)(-\bar zX^\star+M^\star)  =  
 \det(-zU^\star XV + \Sigma)(-\bar zV^\star X^\star U + \Sigma)  
 \stackrel{\mathcal L}{=} |\det(zX + \Sigma)|^2 . 
\]
We also have that the matrix $MM^\star = I - z E - \bar z E^\star + |z|^2 EE^\star$ is
equal to the identity plus a deformation of rank $2r$ at most.  Therefore, the
diagonal $n\times n$ matrix $\Sigma$ of the singular values of $M$ contains
ones on its diagonal except for $2r$ singular values at most. Moreover, using
Assumption~\ref{ass:E} and recalling that $z \in \mathcal D(0,1)$, we obtain
that there exists $C_\Sigma \geq 1$ independent of $n$ and $z$ such that 
$\|\Sigma \|\leq C_\Sigma$.

We now compute $\E |q(z)|^2 = \E |\det(zX + \Sigma)|^2$ where we develop 
$\det(zX + \Sigma)$ using the formula \eqref{detA+B}. Here, we can notice that
$\E \det(zX,\Sigma)_{\cI} \overline{\det(zX,\Sigma)_{\tilde\cI}} = 0$ if 
$\cI \neq \tilde\cI$. Indeed, the case being, one of the matrices 
$(zX,\Sigma)_{\cI}$ or $(zX,\Sigma)_{\tilde\cI}$ contains a column of $zX$ 
that is not present in the other. Making a Laplace expansion of the 
corresponding determinant along this column, we obtain that the cross 
expectation is zero. We therefore get that 
\[
\E |q(z)|^2 = \sum_{k=0}^n 
 \sum_{\cI\in[n] : |\cI| = k} \E \left| \det(zX,\Sigma)_\cI \right|^2. 
\]
Let us work on one of these determinants. For a given $k$, let us assume for
simplicity that $\cI = [k]$. Otherwise, we can permute the rows and columns of 
$(zX,\Sigma)_\cI$ properly; this does not affect 
$\left| \det(zX,\Sigma)_\cI \right|^2$. 
Writing $[k]^{\text{c}} = [n]\setminus [k]$, we have 
\[
\det(zX,\Sigma)_\cI = \det(zX,\Sigma)_{[k]} = 
\det\begin{bmatrix} 
 z X_{[k],[k]} & 0 \\ z X_{[k]^{\text{c}}, [k]} & 
   \Sigma_{[k]^{\text{c}},[k]^{\text{c}}} \end{bmatrix} 
 = z^k \det X_{[k],[k]} \ \det \Sigma_{[k]^{\text{c}},[k]^{\text{c}}} , 
\]
and $\E | \det(zX,\Sigma)_{[k]} |^2 = 
  |z|^{2k} \left| \det \Sigma_{[k]^{\text{c}},[k]^{\text{c}}} \right|^2 
  \E | \det X_{[k],[k]} |^2$. 
By the properties of $\Sigma$ stated above, we have 
$\left| \det \Sigma_{[k]^{\text{c}},[k]^{\text{c}}} \right|^2 \leq 
 C_{\Sigma}^{4r}$. Moreover, if $k = 0$, then  
$\E | \det X_{[k],[k]} |^2 = 1$, otherwise, 
\[
\E | \det X_{[k],[k]} |^2 = 
\sum_{\sigma\in\mathfrak S_k} \E \left| X_{1,\sigma(1)} \ldots 
   X_{k,\sigma(k)} \right|^2 = \frac{k ! }{n^k} .
\]
Therefore, $\E | \det(zX,\Sigma)_{[k]} |^2 \leq |z|^{2k} C_{\Sigma}^{4r} 
 k ! / n^k$, and we end up with 
\[
\E |q(z)|^2 = \sum_{k=0}^n 
 \sum_{\cI\in[n] : |\cI| = k} \E \left| \det(zX,\Sigma)_\cI \right|^2 
 \leq C_{\Sigma}^{4r} \sum_{k=0}^n \binom{n}{k} \frac{k ! }{n^k} |z|^{2k}  
 \leq C_{\Sigma}^{4r} \sum_{k=0}^\infty  |z|^{2k}  
 = \frac{C_{\Sigma}^{4r}}{1 - | z |^2}. 
\]
The tightness of $(q_n)$ follows by applying Proposition \ref{prop:tightness}.
\end{proof}

Moreover it is sufficient to examine the characteristic polynomial of $Y^n$, when the entries of $A^n$ are bounded almost surely. Specifically

\begin{prop}[Truncation]
\label{trunc} 
    Let Assumptions \ref{ass:K} and \ref{ass:E} hold. Let $D > 0$ and define
\[
A^D_{i,j}=A_{i,j}1_{|A_{i,j}|\leq D} -\E A_{i,j}1_{|A_{i,j}|\leq D}\ ,\qquad 
 X^{n,D}_{i,j}=\frac{1}{\sqrt{K_n}}B^n_{i,j}\,A^D_{i,j} \qquad
\text{and}\qquad 
  Y^{n,D}_{ij}=X^{n,D}_{ij}+E^n_{ij}\ , \qquad (i,j\in [n])\, .
\]
Let 
$$
X^{n,D} = \left[ X^{n,D}_{i,j} \right]_{i,j\in[n]}\ ,\qquad Y^{n,D} = X^{n,D} + E^{n}\qquad \textrm{and}\qquad  
q_n^D(z) = \det\left( I_n - z Y^{n,D} \right)\ .
$$ 
Then, 
\[
\forall z\in\mathcal D(0,1), \quad
  \sup_n \E \left| q_n(z) - q_n^D(z) \right|^2 \leq  \varepsilon(D)
 \qquad \textrm{where}\qquad \varepsilon(D)\xrightarrow[D\to\infty]{}0 .
\]
\end{prop}

\begin{proof}
We omit the supersript $n$ in the sequel to lighten the notations. Without
a risk of confusion, we replace, \emph{e.g.}, $Y^{n,D}$ with $Y^D$. We closely
follow the principles and notations introduced in the previous proof. Let 
$M = I - zE$ as above. Writing 
$$
 \E \left| q_n(z) - q_n^D(z) \right|^2 
 =  \E \left| 
\sum_{\sigma \in \mathfrak S_n} \left( \prod_{i=1}^n 
 ( - z X_{i,\sigma(i)} + M_{i,\sigma(i)}) 
-  \prod_{i=1}^n( - z X^D_{i,\sigma(i)} + M_{i,\sigma(i)}) \right)
 \right|^2 
$$
and developing, we notice that $\E \left| q_n(z) - q_n^D(z) \right|^2$ depends on each element $X_{ij}$ via the vector $\E \begin{bmatrix} X_{ij} \\ X_{ij}^D\end{bmatrix}(=0)$ and
the $2\times 2$ matrix 
$$
R_D = n \E \begin{bmatrix} X_{ij} \\ X_{ij}^D\end{bmatrix} 
\begin{bmatrix} \overline{X_{ij}} & \overline{X_{ij}^D}\end{bmatrix}
$$
which does not
depend on $n$. Therefore, we can assume without loss of generality that the vector $\sqrt{n} \begin{bmatrix} X_{ij} \\ X_{ij}^D\end{bmatrix}$ is a cicularly symmetric Gaussian vector (see the definition in \cite{telatar1999capacity} for instance) with covariance matrix $R_D$, and in particular:
$$
\E \begin{bmatrix} X_{ij} \\ X_{ij}^D\end{bmatrix}=0\,,\quad n \E \begin{bmatrix} X_{ij} \\ X_{ij}^D\end{bmatrix} 
\begin{bmatrix} \overline{X_{ij}} & \overline{X_{ij}^D}\end{bmatrix} = R_D\quad \textrm{and}\quad n \E \begin{bmatrix} X_{ij} \\ X_{ij}^D\end{bmatrix} 
\begin{bmatrix} X_{ij} & X_{ij}^D\end{bmatrix} =0\, .
$$
Assuming this, we first observe that $\vect[X \ X_D]$ is a $\C^{2n^2}$--valued
circularly symmetric Gaussian vector, and so is vector 
$A \vect[X \ X_D]$ for any deterministic $p\times 2n^2$ matrix $A$.
Consider now $n\times n$ deterministic matrices $U,V$. Applying 
\cite[Lemma 4.3.1]{horn1994topics} we have
$$
\vect[UXV \ UX^DV] = (V^T\otimes U) \vect[ X \ X^D ],  
$$
hence $\vect [ UXV \ UX^DV]$ is circularly symmetric Gaussian, in particular
$$
\mathbb{E} [UXV]_{ij} [UYV]_{st} = 0 \quad 
 \textrm{for any}\quad i,j,s,t\in [n]\quad \textrm{and} 
  \quad Y\in \{ X, X^D\}\, .
$$
We now wish to understand the covariance structure of the components of
$\vect[ UXV \ \ UX^DV]$ in the case where $U$ and $V$ are unitary. 
\begin{eqnarray*}
    \mathbb{E} [UXV]_{ij} \overline{[UXV]_{st}} &=& \sum_{k,\ell} \sum_{p,q} U_{ik} \overline{U_{sp}} 
    V_{\ell j} \overline{V_{qt}} \mathbb{E} X_{k\ell}\overline{X_{pq}}\ ,\\
    &=& \sum_k U_{ik} \overline{U_{sk}} \sum_{\ell} V_{\ell j} \overline{V_{\ell t}} [R_D]_{11}\ ,\\
    &=&  [UU^\star]_{is} [V^\star V]_{tj} [R_D]_{11}\quad=\quad \delta_{is} \delta_{jt}\ [R_D]_{11}  \qquad \textrm{where}\quad \delta_{ab}=\begin{cases} 1 &\textrm{if}\ a=b\\
    0&\textrm{else} \end{cases}\ .
\end{eqnarray*}
Similarly we can prove that 
\begin{eqnarray*}
    \mathbb{E} [UXV]_{ij} \overline{[UX^DV]_{st}} =\delta_{is} \delta_{jt}\ [R_D]_{12}\qquad \textrm{and}\qquad 
    \mathbb{E} [UX^DV]_{ij} \overline{[UX^DV]_{st}}= \delta_{is} \delta_{jt}\ [R_D]_{22}\ .    
\end{eqnarray*}
Collecting all these properties, we have proved that for any $n\times n$ deterministic, unitary matrices $U,V$, 
$$
[X \  X^D]\ \stackrel{\mathcal L}{=}\ [UXV\ \ UX^DV] . 
$$
Therefore, using the singular value decomposition $M = U \Sigma V^\star$ and 
Equation~\eqref{detA+B}, we have:
\begin{eqnarray*} 
 \E \left| q_n(z) - q_n^D(z) \right|^2 &=& 
 \E \left| \det(-z U^\star X V + \Sigma) - \det(-z U^\star X^D V + \Sigma) \right|^2 \\
 &=& \E \left| \det(z X + \Sigma) - \det(z X^D + \Sigma) \right|^2 \\
 &=& \E\left| \sum_{k=0}^n \sum_{\cI\subset[n]:|\cI|=k} 
   \det(zX,\Sigma)_\cI - \det(zX^D,\Sigma)_\cI \right|^2 \\
 &=& \sum_{k=0}^n \sum_{\cI\subset[n]:|\cI|=k} 
   \E\left| \det(zX,\Sigma)_\cI - \det(zX^D,\Sigma)_\cI \right|^2 \, , 
\end{eqnarray*}
by relying on the fact (established in the previous proof) that 
$$
\E \det(zX,\Sigma)_\cI \overline{\det(zX^D,\Sigma)}_{\tilde\cI} = 0\quad \textrm{if}\quad \cI\neq \tilde\cI\, .
$$ 
Let $\cI = [k]$ as above for some $k\ge 1$, then 
\begin{eqnarray*}
\E\left| \det(zX,\Sigma)_\cI - \det(zX^D,\Sigma)_\cI \right|^2 
& =& |z|^{2k} \E\left| \det X_{[k],[k]} - \det X^D_{[k],[k]} \right|^2 
 \left| \det \Sigma_{[k]^{\text{c}},[k]^{\text{c}}} \right|^2 \ ,\\
 &\leq& 
|z|^{2k} C(\Sigma,r) 
  \E\left| \det X_{[k],[k]} - \det X^D_{[k],[k]} \right|^2\ , 
\end{eqnarray*}
where $C(\Sigma, r)$ is a constant independent of $n$ by 
Assumption \ref{ass:E}. 

Recall that $\mathbb{E} |A_{11}|^2= 1$, notice that 
$
\mathbb{E} |A_{11}^D|^2\le 1$ and
$$
\varepsilon(D):= \mathbb{E} | A_{11} - A_{11}^D|^2 \xrightarrow[D\to\infty]{}0\, .
$$
We have:
\begin{eqnarray*}
\E \left| \prod_{i\in [k]} X_{1i} - \prod_{i \in [k]} X^D_{1i} \right|^2 &=& 
\E\left| (X_{11} - X^D_{11}) X_{12}\cdots X_{1k} + 
    \cdots + 
  X^D_{11} \cdots X^D_{1,k-1} (X_{1k} - X^D_{1k}) \right|^2  \quad \leq\quad \frac{k}{n^k} \varepsilon(D) 
\end{eqnarray*}
by Minkowski's inequality. We therefore obtain that 
\[
\E\left| \det X_{[k],[k]} - \det X^D_{[k],[k]} \right|^2 \quad =\quad  
 \E\left| \sum_{\sigma\mathfrak S_k} X_{1\sigma(1)} \cdots X_{k\sigma(k)} 
  - X^D_{1\sigma(1)} \cdots X^D_{k\sigma(k)} \right|^2 
 \quad \leq\quad   k! \,\frac{k} {n^k}\, \varepsilon(D)\,.
\]
Now,
$$
\sum_{\cI\subset[n]:|\cI|=k} 
   \E\left| \det(zX,\Sigma)_\cI - \det(zX^D,\Sigma)_\cI \right|^2\quad \le \quad \binom{n}{k}\   k! \,\frac{k} {n^k}\, \varepsilon(D)\, ,
$$
and finally 
\[
 \E \left| q_n(z) - q_n^D(z) \right|^2 \quad \leq\quad  
C(\Sigma,r) \ \varepsilon(D) \ \sum_{k=0}^\infty k |z|^{2k} \quad \leq\quad 
 C \varepsilon_D\, . 
\]
The proposition is proven. 
\end{proof}

\subsection{Moments of $Y^n$ and $X^n$}

We study the asymptotic behavior of the vector
\[
\bigl(1,\tr(Y^n),\dots,\tr((Y^n)^k)\bigr),
\qquad k\in\N.
\]
Throughout, we write $\tr((Y^n)^k)$ (and similarly for $X^n,E^n$); this is the quantity expanded below.

\textbf{Circles}: We consider directed circles (called \emph{circles}) consisting of exactly $k$ edges. In our setting, a \emph{circle}
is an Eulerian cycle of a strongly connected directed multigraph; vertices may repeat and multiple edges (including
loops and parallel edges) are allowed. We identify underlying strongly connected directed multigraphs up to graph
isomorphism, and denote by $\mathcal{C}_k$ the collection of all Eulerian cycles of length $k$ arising from all such
isomorphism classes.

Formally, an element $\mathbf{C}\in\mathcal{C}_k$ can be represented by a cyclic sequence
\[
\mathbf{C}=\{u_1,u_2,\dots,u_k,u_1\},
\]
where the vertices $u_i$ are not necessarily distinct, and each consecutive pair $(u_i,u_{i+1})$ forms a directed edge
(with the convention $u_{k+1}=u_1$). We denote by $V(\mathbf{C})$ the set of vertices appearing in $\mathbf{C}$, and by
\[
E(\mathbf{C})=\bigl\{(u_i,u_{i+1}) : i=1,\dots,k\bigr\}
\]
the multiset of edges. For an edge $e=(u,v)\in E(\mathbf{C})$, its \emph{multiplicity} is
\[
\bigl|\{\tilde e\in E(\mathbf{C}) : \tilde e=e\}\bigr|.
\]
We call $u$ the \emph{source} of $e$ and $v$ its \emph{target}.

For any $\mathbf{C}\in\mathcal{C}_k$ and $B\subset E(\mathbf{C})$, we denote by $\mathbf{C}\setminus B$ the directed
multigraph with edge multiset $E(\mathbf{C})\setminus B$ and vertex set induced by these edges.

\textbf{Labelings.}
Given $B\subset E(\mathbf{C})$ and a labeling $\mathbf{i}\in[n]^{|V(\mathbf{C})|}$, we write $\mathbf{i}\sim\mathbf{C}$
if $\mathbf{i}$ assigns distinct indices from $[n]$ to the vertices of $\mathbf{C}$ according to their first order of
appearance along the circle. We denote by $\mathbf{i}(B)$ the (multi)set of labeled edges corresponding to $B$.

With this notation,
\begin{align*}
\tr\!\bigl((Y^n)^k\bigr)
&=\sum_{(i_1,\dots,i_k)\in[n]^k}\prod_{\ell=1}^k (X^{n}+E^n)_{i_\ell,i_{\ell+1}}  \\
&=\sum_{\mathbf{C}\in\mathcal{C}_k}\ \sum_{B\subset E(\mathbf{C})}\
\sum_{\substack{\mathbf{i}\sim\mathbf{C}\\ \mathbf{i}\in[n]^{|V(\mathbf{C})|}}}
\ \prod_{(i,j)\in\mathbf{i}(B)}E^n_{i,j}\ \prod_{(i,j)\notin\mathbf{i}(B)}X^n_{i,j},
\end{align*}
under the convention $i_{k+1}=i_1$. Therefore,
\begin{align}\label{expected_value}
\tr\!\bigl((Y^n)^k\bigr)-\tr\!\bigl((X^n)^k\bigr)-\tr\!\bigl((E^n)^k\bigr)
=\sum_{\mathbf{C}\in\mathcal{C}_k}\ \sum_{\substack{B\subset E(\mathbf{C})\\ B\neq\emptyset,\ E(\mathbf{C})}}
\sum_{\substack{\mathbf{i}\sim\mathbf{C}\\ \mathbf{i}\in[n]^{|V(\mathbf{C})|}}}
\ \prod_{(i,j)\in\mathbf{i}(B)}E^n_{i,j}\ \prod_{(i,j)\notin\mathbf{i}(B)}X^n_{i,j}.
\end{align}

\subsubsection*{Auxiliary notation}

\begin{notation}\label{not:bd-deg}
\begin{itemize}
\item For any finite multiset $A$, we denote by $|A|_{\mathrm{no}}$ the cardinality of its underlying set (i.e., ignoring
multiplicities).
\item For $\mathbf{C}\in\mathcal{C}_k$ and $B\subset E(\mathbf{C})$, define
\[
E\bigl(\bd(\mathbf{C}\setminus B)\bigr)
=\Bigl\{e\in B:\ \exists\, v\in V(\mathbf{C}\setminus B)\ \text{such that $v$ is incident to $e$}\Bigr\}.
\]
Then $\bd(\mathbf{C}\setminus B)$ denotes the directed multigraph induced by the edge multiset
$E(\bd(\mathbf{C}\setminus B))$.
\item For $\mathbf{C}\in\mathcal{C}_k$ and $B\subset E(\mathbf{C})$, let $\mathbf{C}_B$ be the directed multigraph induced
by the edge multiset $B$. We write $\mathbf{i}\sim\mathbf{C}_B$ to indicate a labeling $\mathbf{i}\in[n]^{|V(\mathbf{C}_B)|}$
assigning distinct values to the vertices of $\mathbf{C}_B$.
\item For any directed multigraph $G$ and $v\in V(G)$, we denote by $\deg_G^+(v)$ (resp.\ $\deg_G^-(v)$) the out-degree
(resp.\ in-degree) of $v$, counted with multiplicity.
\end{itemize}
\end{notation}

\subsubsection*{A combinatorial inequality}

\begin{lem}\label{|B|-tree}
Fix $\mathbf{C}\in\mathcal{C}_k$ and $B\subset E(\mathbf{C})$ such that $B\neq\emptyset$ and $B\neq E(\mathbf{C})$.
Assume:
\begin{enumerate}
\item $\mathbf{C}\setminus B$ is weakly connected;
\item for every $e\in E(\mathbf{C}\setminus B)$, the multiplicity
\(
|\{\tilde e\in E(\mathbf{C}):\tilde e=e\}|\ge 2.
\)
\end{enumerate}
Then
\begin{align*}
&|V(\mathbf{C} \setminus  B)|
- \left|\left\{v \in V(\mathbf{C} \setminus  B)\cap V(\bd(\mathbf{C} \setminus  B)):
\deg^+_{\mathbf{C}_B}(v) + \deg^-_{\mathbf{C}_B}(v) \geq 2\right\}\right| \\
&\quad - \frac{1}{2} \left|\left\{v \in V(\mathbf{C}_B):
\deg^+_{\mathbf{C}_B}(v) + \deg^-_{\mathbf{C}_B}(v) = 1\right\}\right|
- |E(\mathbf{C}\setminus B)|_{\mathrm{no}}
\le -1.
\end{align*}
\end{lem}

\begin{proof}
Since $\mathbf{C}\setminus B$ is weakly connected, one has the standard bound
\[
|V(\mathbf{C}\setminus B)|\le |E(\mathbf{C}\setminus B)|_{\mathrm{no}}+1.
\]
We distinguish the following cases:
\begin{enumerate}
\item\label{case:tree} $|V(\mathbf{C}\setminus B)| = |E(\mathbf{C}\setminus B)|_{\mathrm{no}}+1$;
\item\label{case:unicyc-nobd2} $|V(\mathbf{C}\setminus B)| = |E(\mathbf{C}\setminus B)|_{\mathrm{no}}$ and
\[
\left|\left\{v \in V(\mathbf{C} \setminus  B)\cap V(\bd(\mathbf{C} \setminus  B)):
\deg^+_{\mathbf{C}_B}(v) + \deg^-_{\mathbf{C}_B}(v) \geq 2\right\}\right| = 0;
\]
\item\label{case:unicyc-bd2} $|V(\mathbf{C}\setminus B)| = |E(\mathbf{C}\setminus B)|_{\mathrm{no}}$ and
\[
\left|\left\{v \in V(\mathbf{C} \setminus  B)\cap V(\bd(\mathbf{C} \setminus  B)):
\deg^+_{\mathbf{C}_B}(v) + \deg^-_{\mathbf{C}_B}(v) \geq 2\right\}\right| \ge 1;
\]
\item\label{case:moreedges} $|V(\mathbf{C}\setminus B)| < |E(\mathbf{C}\setminus B)|_{\mathrm{no}}$.
\end{enumerate}
Cases \ref{case:unicyc-bd2} and \ref{case:moreedges} are immediate from the definition of the left-hand side, so we
treat \ref{case:tree} and \ref{case:unicyc-nobd2}.

\medskip\noindent
\emph{Case \ref{case:tree}.}
Here the underlying simple undirected graph of $\mathbf{C}\setminus B$ is a tree. Define
\begin{align*}
V^+(\mathbf{C}\setminus B)
&=\{v\in V(\mathbf{C}\setminus B):\exists\, (v,a)\in E(\mathbf{C}\setminus B)\},\\
V^-(\mathbf{C}\setminus B)
&=\{v\in V(\mathbf{C}\setminus B):\exists\, (a,v)\in E(\mathbf{C}\setminus B)\}.
\end{align*}
Since $\mathbf{C}\setminus B$ is a finite tree, there exist vertices $w\notin V^+(\mathbf{C}\setminus B)$ and
$u\notin V^-(\mathbf{C}\setminus B)$; otherwise, starting from any vertex one could construct an infinite directed path,
contradicting finiteness.
Each of $u,w$ is incident to at least one edge of $\mathbf{C}\setminus B$, and by assumption those edges have multiplicity
at least $2$ in $\mathbf{C}$. Because $\mathbf{C}$ is a circle, there are at least two directed edges in $\mathbf{C}$
leaving $w$ and at least two entering $u$. By the choice of $u,w$, these additional edges must belong to $B$, yielding
the required inequality.

\medskip\noindent
\emph{Case \ref{case:unicyc-nobd2}.}
Since $\mathbf{C}$ is a circle and no boundary vertex has total degree at least $2$ in $\mathbf{C}_B$, we have
\begin{align}\label{entering/going_edges}
\left| V(\mathbf{C}\setminus B)\cap V(\bd(\mathbf{C}\setminus B))\right|
&=
\left|\left\{v \in V(\mathbf{C}_B):
\deg^+_{\mathbf{C}_B}(v) + \deg^-_{\mathbf{C}_B}(v) = 1\right\}\right| \\
&=
\left| E(\bd(\mathbf{C}\setminus B))  \right|
\ge 2. \nonumber
\end{align}
Indeed, if \eqref{entering/going_edges} failed, then $\mathbf{C}$ either could not enter or could not exit
$\mathbf{C}\setminus B$, contradicting that $\mathbf{C}$ is a circle. The claim follows.
\end{proof}

\begin{rem}
If $\mathbf{C}\setminus B$ has several weakly connected components, Lemma~\ref{|B|-tree} applies to each component
separately.
\end{rem}

\subsubsection*{Main combinatorial consequence}

\begin{prop}\label{tr(Y)-tr(X)-det}
Assume that $|A_{1,1}^n|\le D$ for some (fixed) constant $D>0$. Then for every $k\in\N$,
\[
\tr\!\bigl((Y^n)^k\bigr)-\tr\!\bigl((X^n)^k\bigr)-\tr\!\bigl((E^n)^k\bigr)
\ \xrightarrow[n\to\infty]{\ \mathbb{P}\ }\ 0.
\]
\end{prop}

\begin{proof}
We prove the claim by controlling the mean and the variance.

\medskip\noindent
\textbf{Step 1: bound on the mean.}
We show that
\begin{align}\label{bound_on_mean}
\E\Bigl[\tr\!\bigl((Y^n)^k\bigr)-\tr\!\bigl((X^n)^k\bigr)-\tr\!\bigl((E^n)^k\bigr)\Bigr]
=O\!\left(\frac1n\right).
\end{align}
By \eqref{expected_value}, it suffices to bound
\begin{align}\label{o(1)_expectation}
\sum_{\mathbf{C} \in \mathcal{C}_k}
\sum_{\substack{B \subset E(\mathbf{C})\\ B \neq E(\mathbf{C}),\ \emptyset}}
\sum_{\substack{\mathbf{i}\sim \mathbf{C}\\ \mathbf{i}\in [n]^{|V(\mathbf{C})|}}}
\left| \prod_{(i,j) \in \mathbf{i}(B)} E^n_{i,j} \right|
\left| \E\prod_{(i,j) \notin \mathbf{i}(B)} X_{i,j}^n \right|
= O\!\left( \frac{1}{n} \right).
\end{align}
For the expectation in \eqref{o(1)_expectation} to be non-zero, every edge of $\mathbf{C}\setminus B$ must have
multiplicity at least $2$.

Since the number of circles in $\mathcal{C}_k$ and the number of subsets $B\subset E(\mathbf{C})$ depend only on $k$, it
is enough to fix $\mathbf{C}\in\mathcal{C}_k$ and a non-trivial $B\subset E(\mathbf{C})$ and prove
\begin{align}\label{bound_on_terms_mean_value}
\sum_{\substack{\mathbf{i}\sim \mathbf{C}\\ \mathbf{i}\in [n]^{|V(\mathbf{C})|}}}
\left| \prod_{(i,j) \in \mathbf{i}(B)} E^n_{i,j}\right|
\left| \E\prod_{(i,j) \notin \mathbf{i}(B)} X^n_{i,j} \right|
= O\!\left( \frac{1}{n} \right),
\end{align}
under the assumption that every edge in $\mathbf{C}\setminus B$ has multiplicity at least $2$. We also assume for
simplicity that $\mathbf{C}\setminus B$ is weakly connected (the case of several weakly connected components is treated
component-wise).

Since the entries of $A^n$ are bounded by $D$, we obtain
\begin{align*}
&\sum_{\substack{\mathbf{i}\sim \mathbf{C}\\ \mathbf{i}\in [n]^{|V(\mathbf{C})|}}}
\left| \prod_{(i,j) \in \mathbf{i}(B)} E^n_{i,j} \right|
\left| \E\prod_{(i,j) \notin \mathbf{i}(B)} X^n_{i,j} \right| \\
&\qquad \le D^{2k}
\left( \frac{1}{\sqrt{K_n}} \right)^{|E(\mathbf{C} \setminus  B)|}
\left( \frac{K_n}{n} \right)^{|E(\mathbf{C} \setminus  B)|_{\mathrm{no}}}
\sum_{\substack{\mathbf{i}\sim \mathbf{C}\\ \mathbf{i}\in [n]^{|V(\mathbf{C})|}}}
\left| \prod_{(i,j) \in \mathbf{i}(B)} E^n_{i,j} \right|.
\end{align*}
Since each edge of $\mathbf{C}\setminus B$ appears at least twice, we have
\begin{align}\label{mutliedge}
|E(\mathbf{C} \setminus  B)|_{\mathrm{no}} \le \frac{|E(\mathbf{C} \setminus  B)|}{2}.
\end{align}
Moreover, connectivity of the underlying simple graph yields
\begin{align}\label{vertices_edges}
|V(\mathbf{C} \setminus  B)| \le |E(\mathbf{C} \setminus  B)|_{\mathrm{no}} + 1.
\end{align}

It remains to bound
\[
\sum_{\substack{\mathbf{i}\sim \mathbf{C}\\ \mathbf{i}\in [n]^{|V(\mathbf{C})|}}}
\left| \prod_{(i,j) \in \mathbf{i}(B)} E^n_{i,j} \right|.
\]
Using the representation of $E^n$ and the entrywise bound
\[
|E^n_{i,j}|\le r\ \max_{\ell\in[r]}|u^{\ell,n}_i|\ \max_{\ell\in[r]}|v^{\ell,n}_j|,
\]
we obtain (as in \eqref{ineq_for_deterministic} in the original derivation)
\begin{align}\label{ineq_for_deterministic}
& \sum_{\substack{\mathbf{i}\sim \mathbf{C}\\ \mathbf{i}\in [n]^{|V(\mathbf{C})|}}}
\left| \prod_{(i,j) \in \mathbf{i}(B)} E^n_{i,j} \right|
\le r^k\, n^{|V(\mathbf{C} \setminus  B)| - |V(\mathbf{C} \setminus  B) \cap V(\bd(\mathbf{C} \setminus  B))|}
\prod_{v \in V(\mathbf{C}_B)}
\left(
\sum_{i \in [n]}
\max_{\ell \in [r]}|v^{\ell,n}_i|^{\deg^-_{\mathbf{C}_B}(v)}
\max_{\ell \in [r]}|u^{\ell,n}_i|^{\deg^+_{\mathbf{C}_B}(v)}
\right).
\end{align}
Now, for each $v\in V(\mathbf{C}_B)$:
\begin{itemize}
\item If $\deg^+_{\mathbf{C}_B}(v)+\deg^-_{\mathbf{C}_B}(v)\ge 2$, then by Assumption~\ref{ass:E} (and the same
Cauchy--Schwarz argument as in the original proof),
\[
\sum_{i \in [n]}
\max_{\ell \in [r]}|v^{\ell,n}_i|^{\deg^-_{\mathbf{C}_B}(v)}
\max_{\ell \in [r]}|u^{\ell,n}_i|^{\deg^+_{\mathbf{C}_B}(v)}
\le C
\]
for some constant $C>0$.
\item If $\deg^+_{\mathbf{C}_B}(v)+\deg^-_{\mathbf{C}_B}(v)=1$, then by Cauchy--Schwarz,
\[
\sum_{i\in[n]}
\max_{\ell\in[r]}|v^{\ell,n}_i|^{\deg^+_{\mathbf{C}_B}(v)}
\max_{\ell\in[r]}|u^{\ell,n}_i|^{\deg^-_{\mathbf{C}_B}(v)}
\le rC\sqrt{n}.
\]
\end{itemize}
Thus, for some $C=C(k,r)>0$,
\begin{align*}
\sum_{\substack{\mathbf{i}\sim \mathbf{C}\\ \mathbf{i}\in [n]^{|V(\mathbf{C})|}}}
\left| \prod_{(i,j) \in \mathbf{i}(B)} E^n_{i,j} \right|
\le
C\,
n^{\,|V(\mathbf{C} \setminus  B)|
- |V(\mathbf{C} \setminus  B)\cap V(\bd(\mathbf{C} \setminus  B))|
+\frac12\left|\{v \in V(\mathbf{C}_B): \deg^+_{\mathbf{C}_B}(v)+\deg^-_{\mathbf{C}_B}(v)=1\}\right| }.
\end{align*}
Because $\mathbf{C}$ is a circle, any vertex $v$ with $\deg^+_{\mathbf{C}_B}(v)=0$ or $\deg^-_{\mathbf{C}_B}(v)=0$ must
lie in $V(\mathbf{C}\setminus B)\cap V(\bd(\mathbf{C}\setminus B))$. Hence, setting
\begin{align*}
a(\mathbf{C}\setminus B)
&:=|V(\mathbf{C} \setminus  B)|
- \left|\left\{v \in V(\mathbf{C} \setminus  B)\cap V(\bd(\mathbf{C} \setminus  B)):
\deg^+_{\mathbf{C}_B}(v) + \deg^-_{\mathbf{C}_B}(v) \geq 2\right\}\right| \\
&\qquad\qquad
- \frac{1}{2} \left|\left\{v \in V(\mathbf{C}_B):
\deg^+_{\mathbf{C}_B}(v) + \deg^-_{\mathbf{C}_B}(v) = 1\right\}\right|,
\end{align*}
we conclude that for a constant $C_k$ independent of $n$,
\begin{align}\label{last_bound_mean_value}
&\sum_{\substack{\mathbf{i}\sim \mathbf{C}\\ \mathbf{i}\in [n]^{|V(\mathbf{C})|}}}
\left| \prod_{(i,j) \in \mathbf{i}(B)} E^n_{i,j} \right|
\left| \E \prod_{(i,j) \notin \mathbf{i}(B)} X^n_{i,j} \right| \nonumber \\
&\qquad \le
C_k\,
\left(K_n\right)^{|E(\mathbf{C} \setminus  B)|_{\mathrm{no}} - \frac{|E(\mathbf{C} \setminus  B)|}{2}}
\,n^{a(\mathbf{C}\setminus B)-|E(\mathbf{C} \setminus  B)|_{\mathrm{no}}}.
\end{align}
The desired $O(1/n)$ bound follows from Lemma~\ref{|B|-tree} together with \eqref{mutliedge}. This proves
\eqref{bound_on_mean}.

\medskip\noindent
\textbf{Step 2: bound on the variance.}
We show that
\begin{align}\label{bound_on_var}
\Var\Bigl(\tr\!\bigl((Y^n)^k\bigr)-\tr\!\bigl((X^n)^k\bigr)-\tr\!\bigl((E^n)^k\bigr)\Bigr)
=O\!\left(\frac{1}{nK_n}\right).
\end{align}
Recall that for complex random variables $\{W_i\}_{i=1}^m$,
\[
\Var\Bigl(\sum_{i=1}^m W_i\Bigr)
=\sum_{i_1,i_2\in[m]}
\E\Bigl[\bigl(W_{i_1}-\E W_{i_1}\bigr)\overline{\bigl(W_{i_2}-\E W_{i_2}\bigr)}\Bigr].
\]
Applying this to the expansion \eqref{expected_value} yields
\begin{align}
&\Var\Bigl(\tr\!\bigl((Y^n)^k\bigr)-\tr\!\bigl((X^n)^k\bigr)-\tr\!\bigl((E^n)^k\bigr)\Bigr)
= \sum_{\mathbf{C},\mathbf{C'} \in \mathcal{C}_k}
\sum_{\substack{B \subset E(\mathbf{C})\\ B \neq \emptyset,\ E(\mathbf{C})}}
\sum_{\substack{B' \subset E(\mathbf{C'})\\ B' \neq \emptyset,\ E(\mathbf{C'})}}
\sum_{\substack{\mathbf{i}\sim \mathbf{C}\\ \mathbf{i}\in [n]^{|V(\mathbf{C})|}}}
\sum_{\substack{\mathbf{i'}\sim \mathbf{C'}\\ \mathbf{i'}\in [n]^{|V(\mathbf{C'})|}}}
\cdot \nonumber\\
&\qquad\qquad
\prod_{(i,j) \in \mathbf{i}(B)} E^n_{i,j}\,
\prod_{(i',j') \in \mathbf{i'}(B')} \overline{ E^n_{i',j'}}\,
\E \Bigl( \prod_{(i,j) \notin \mathbf{i}(B)} X^n_{i,j} - \E \prod_{(i,j) \notin \mathbf{i}(B)} X^n_{i,j} \Bigr)
\nonumber\\
&\qquad\qquad\qquad\qquad\qquad\times
\Bigl( \prod_{(i',j') \notin \mathbf{i'}(B')} \overline{X^n_{i',j'}} - \E \prod_{(i',j') \notin \mathbf{i'}(B')} \overline{X^n_{i',j'}} \Bigr).
\label{random_term_in_var}
\end{align}
By independence of the entries of $X^n$, the expectation in \eqref{random_term_in_var} vanishes unless the labeled edge
sets in $\mathbf{C}\setminus B$ and $\mathbf{C'}\setminus B'$ agree on at least one edge, and every edge in the multigraph
$\mathbf{C}\setminus B\ \cup\ \mathbf{C'}\setminus B'$ has multiplicity at least $2$. Hence, it suffices to show that
for any such $\mathbf{C},\mathbf{C'}$ and non-trivial $B,B'$,
\begin{align}\label{covariance}
&\sum_{\substack{\mathbf{i}\sim \mathbf{C}\\ \mathbf{i}\in [n]^{|V(\mathbf{C})|}}}
\sum_{\substack{\mathbf{i'}\sim \mathbf{C'}\\ \mathbf{i'}\in [n]^{|V(\mathbf{C'})|}}}
\left| \prod_{(i,j) \in \mathbf{i}(B)} E^n_{i,j}
\prod_{(i',j') \in \mathbf{i'}(B')} \overline{E^n_{i',j'}}\right|
\cdot \nonumber\\
&\qquad\qquad\times
\left| \E \Bigl( \prod_{(i,j) \notin \mathbf{i}(B)} X^n_{i,j} - \E \prod_{(i,j) \notin \mathbf{i}(B)} X^n_{i,j} \Bigr)
\Bigl( \prod_{(i',j') \notin \mathbf{i'}(B')} \overline{X^n_{i',j'}} - \E \prod_{(i',j') \notin \mathbf{i'}(B')} \overline{X^n_{i',j'}} \Bigr) \right|
= O\!\left(\frac{1}{nK_n}\right).
\end{align}

For simplicity, assume $\mathbf{C}\setminus B$ and $\mathbf{C'}\setminus B'$ are weakly connected; the general case follows
by decomposing into weakly connected components. Since $\mathbf{C}\setminus B$ and $\mathbf{C'}\setminus B'$ share an edge,
the union $\mathbf{C}\setminus B\cup \mathbf{C'}\setminus B'$ is also weakly connected.

Set $\widetilde{\mathbf{C}}:=\mathbf{C}\cup \mathbf{C'}$. After an appropriate ordering of vertices,
$\widetilde{\mathbf{C}}$ defines a circle of length $2k$. Let $\widetilde{\mathbf{C}}_B$ denote the subgraph induced by
the edge multiset $B\cup B'$.

Using the definition of $X^n$ (cf.\ \eqref{sparser_matrices}) and the independence structure, together with the bound
$|A^n_{ij}|\le D$, we have
\[
\left| \E \Bigl( \prod_{(i,j) \notin \mathbf{i}(B)} X^n_{i,j} - \E \prod_{(i,j) \notin \mathbf{i}(B)} X^n_{i,j} \Bigr)
\Bigl( \prod_{(i',j') \notin \mathbf{i'}(B')} \overline{X^n_{i',j'}} - \E \prod_{(i',j') \notin \mathbf{i'}(B')} \overline{X^n_{i',j'}} \Bigr) \right|
\le
2 D^{2k} \left( \frac{K_n}{n} \right)^{|E(\mathbf{C} \setminus  B \cup \mathbf{C}' \setminus B')|_{\mathrm{no}}}.
\]
Moreover, as in \eqref{ineq_for_deterministic}, one shows that
\[
\sum_{\substack{\mathbf{i}\sim \mathbf{C}\\ \mathbf{i}\in [n]^{|V(\mathbf{C})|}}}
\sum_{\substack{\mathbf{i'}\sim \mathbf{C'}\\ \mathbf{i'}\in [n]^{|V(\mathbf{C'})|}}}
\left| \prod_{(i,j) \in \mathbf{i}(B)} E^n_{i,j}\prod_{(i',j') \in \mathbf{i'}(B')} \overline{E^n_{i',j'}}\right|
\le C\, n^{a(\mathbf{C}\setminus B,\mathbf{C'}\setminus B')},
\]
where
\begin{align*}
a(\mathbf{C}\setminus B,\mathbf{C'}\setminus B')
&=
|V(\mathbf{C}\setminus B)\cup V(\mathbf{C'}\setminus B')|
-\left|\left\{v:\deg^+_{\widetilde{\mathbf{C}}_B}(v)+\deg^-_{\widetilde{\mathbf{C}}_B}(v)\ge 2\right\}\right| \\
&\qquad
-\frac12\left|\left\{v:\deg^+_{\widetilde{\mathbf{C}}_B}(v)+\deg^-_{\widetilde{\mathbf{C}}_B}(v)= 1\right\}\right|.
\end{align*}
Since $\widetilde{\mathbf{C}}$ is a circle and every edge in $\mathbf{C}\setminus B\cup \mathbf{C'}\setminus B'$ has
multiplicity at least $2$, Lemma~\ref{|B|-tree} applies and yields (for some $C_k$ depending only on $k,M,r$)
\[
\eqref{covariance}
\le
\frac{C_k}{n}\cdot
\frac{1}{\sqrt{K_n^{|E(\mathbf{C}\setminus B)|+|E(\mathbf{C'}\setminus B')|}}}
\cdot
K_n^{|E(\mathbf{C}\setminus B\cup \mathbf{C'}\setminus B')|_{\mathrm{no}}}.
\]
Finally, since all edges in $\mathbf{C}\setminus B\cup \mathbf{C'}\setminus B'$ appear at least twice and the two graphs
share an edge, we have
\[
|E(\mathbf{C}\setminus B\cup \mathbf{C'}\setminus B')|_{\mathrm{no}}
\le
2\bigl(|E(\mathbf{C}\setminus B)|+|E(\mathbf{C'}\setminus B')|-1\bigr),
\]
and therefore
\[
\eqref{covariance}\le C_k\,\frac{1}{nK_n}.
\]
This proves \eqref{bound_on_var}.

\medskip
Combining \eqref{bound_on_mean} and \eqref{bound_on_var} gives
\(
\tr((Y^n)^k)-\tr((X^n)^k)-\tr((E^n)^k)\to 0
\)
in probability, completing the proof.
\end{proof}

We continue with the asymptotic analysis of the joint law of the traces
$\tr((X^n)^k)$. Recall the notation from Theorem~\ref{characteristic_poly_thm}
and define, for $k\in\N$, the sequence
\[
\mathrm{mean}_k
:= \mathbf{1}_{\{k \text{ even}\}}\bigl(\E A_{1,1}^2\bigr)^{k/2}.
\]

\begin{lem}\label{lemma_moments_of_X^n}
For any $k\ge 1$, if $|A_{1,1}^n|\le D$ almost surely, then
\[
\bigl(\tr(X^n),\ldots,\tr((X^n)^k)\bigr)
\xrightarrow[n\to\infty]{\mathrm{law}}
\bigl(Z_1+\mathrm{mean}_1,\ \sqrt{2}\,Z_2+\mathrm{mean}_2,\ \ldots,\ \sqrt{k}\,Z_k+\mathrm{mean}_k\bigr).
\]
\end{lem}

\begin{proof}
When $K_n\ge \log n$, the claim follows directly from
Propositions~2.3 and~3.6 of~\cite{hachem2025spectral}, applied to our model.

In general, the proof is analogous to that of Lemmas~3.4 and~3.5
in~\cite{bordenave2022convergence}. For completeness, we sketch the main steps.

Recall the notation $\mathcal{C}_k$ for the collection of directed circles of
length $k$. Let $k_1,\ldots,k_m\in\N$, let
$\mathbf{C}_\ell\in\mathcal{C}_{k_\ell}$ for $\ell=1,\ldots,m$, and let
$s_1,\ldots,s_m\in\{\cdot,*\}$, where for any complex number $x$ we set
$x^{\cdot}=x$ and $x^*=\overline{x}$. Define the multigraph
\[
\widetilde{\mathbf{C}} := \bigcup_{\ell=1}^m \mathbf{C}_\ell .
\]
Then the joint contribution of these circles satisfies
\begin{align}\label{joint_Xn}
\sum_{\mathbf{i}\sim\widetilde{\mathbf{C}}}
\left|
\E \prod_{\ell=1}^m
\prod_{(v,u)\in E(\mathbf{C}_\ell)}
\bigl(X^n_{\mathbf{i}(v),\mathbf{i}(u)}\bigr)^{s_\ell}
\right|
\le
D^{\sum_{\ell=1}^m k_\ell}
K_n^{|E(\widetilde{\mathbf{C}})|_{\mathrm{no}}
-\frac{1}{2}\sum_{\ell=1}^m |E(\mathbf{C}_\ell)|}
\,n^{|V(\widetilde{\mathbf{C}})|-|E(\widetilde{\mathbf{C}})|_{\mathrm{no}}}.
\end{align}
Since the entries of $X^n$ are centered, the contribution in
\eqref{joint_Xn} is negligible unless
\[
|V(\widetilde{\mathbf{C}})| = |E(\widetilde{\mathbf{C}})|_{\mathrm{no}}
\qquad\text{and}\qquad
2\,|E(\widetilde{\mathbf{C}})|_{\mathrm{no}}
= \sum_{\ell=1}^m |E(\mathbf{C}_\ell)|.
\]

We proceed as in~\cite{bordenave2022convergence}.
Decompose $\mathcal{C}_k$ as $\mathcal{C}_k=\mathcal{C}_k^1\cup\mathcal{C}_k^2$,
where $\mathcal{C}_k^1$ consists of circles with exactly $k$ distinct vertices,
and $\mathcal{C}_k^2$ consists of circles with fewer than $k$ vertices.
Accordingly, we write
\begin{align*}
\tr((X^n)^k)
&=
\sum_{\mathbf{C}\in\mathcal{C}_k^1}
\sum_{\substack{\mathbf{i}\sim\mathbf{C}\\ \mathbf{i}\in[n]^k}}
\prod_{(v,u)\in E(\mathbf{C})}
X^n_{\mathbf{i}(v),\mathbf{i}(u)}
+
\sum_{\mathbf{C}\in\mathcal{C}_k^2}
\sum_{\substack{\mathbf{i}\sim\mathbf{C}\\ \mathbf{i}\in[n]^k}}
\prod_{(v,u)\in E(\mathbf{C})}
X^n_{\mathbf{i}(v),\mathbf{i}(u)} \\
&=: t_n^k + r_n^k .
\end{align*}

The proof is complete once we establish the following two facts.

\begin{enumerate}
\item
For any $k_1,\ldots,k_m\in\N$ and $s_1,\ldots,s_m\in\{\cdot,*\}$,
\begin{align}\label{joint_law_circles=k}
\E \prod_{\ell=1}^m (t_n^{k_\ell})^{s_\ell}
\ \xrightarrow[n\to\infty]{}\ 
\E \prod_{\ell=1}^m \bigl(\sqrt{k_\ell}\,Z_{k_\ell}\bigr)^{s_\ell}.
\end{align}

\item
For any $k\in\N$,
\begin{align}\label{joint_law_circles<k}
r_n^k \xrightarrow[n\to\infty]{\mathbb{P}} \mathrm{mean}_k .
\end{align}
\end{enumerate}

Given the bound \eqref{joint_Xn}, the convergence
\eqref{joint_law_circles=k} and \eqref{joint_law_circles<k} follow exactly as in
the proofs of Lemmas~3.4 and~3.5 of~\cite{bordenave2022convergence}, respectively.
\end{proof}

We conclude with an asymptotic bound on $\E|\tr((Y^n)^k)|^2$, which is needed to
establish relative compactness.

\begin{lem}\label{trace_bounded_Y^n_lem}
For any $k\in\N$, if $|A_{1,1}^n|\le D$, then there exists a constant
$C=C(r,k)>0$ such that
\[
\E\bigl|\tr((Y^n)^k)\bigr|^2 \le C .
\]
\end{lem}

\begin{proof}
We begin with the expansion
\begin{align}
\E\bigl|\tr((Y^n)^k)\bigr|^2
&=
\sum_{\mathbf{C},\mathbf{C}'\in\mathcal{C}_k}
\sum_{\substack{B\subset E(\mathbf{C})\\ B\neq\emptyset,E(\mathbf{C})}}
\sum_{\substack{B'\subset E(\mathbf{C}')\\ B'\neq\emptyset,E(\mathbf{C}')}}
\sum_{\substack{\mathbf{i}\sim\mathbf{C}\\ \mathbf{i}\in[n]^{|V(\mathbf{C})|}}}
\sum_{\substack{\mathbf{i}'\sim\mathbf{C}'\\ \mathbf{i}'\in[n]^{|V(\mathbf{C}')|}}}
\cdots \nonumber\\
&\qquad\qquad
\prod_{(i,j)\in B} E^n_{i,j}
\prod_{(i',j')\in B'} E^n_{i',j'}
\,
\E\!\left[
\prod_{(i,j)\notin\mathbf{i}(B)} X^n_{i,j}
\prod_{(i',j')\notin\mathbf{i}'(B')}
\overline{X^n_{i',j'}}
\right].
\label{mean_value_bound}
\end{align}

We proceed as in the proof of Proposition~\ref{tr(Y)-tr(X)-det}.
Fix $\mathbf{C},\mathbf{C}'\in\mathcal{C}_k$ and non-trivial subsets
$B\subset E(\mathbf{C})$, $B'\subset E(\mathbf{C}')$.
It suffices to show that
\begin{align}\label{sum_bound_EtrY^b}
\sum_{\substack{\mathbf{i}\sim\mathbf{C}\\ \mathbf{i}\in[n]^{|V(\mathbf{C})|}}}
\sum_{\substack{\mathbf{i}'\sim\mathbf{C}'\\ \mathbf{i}'\in[n]^{|V(\mathbf{C}')|}}}
\left|
\prod_{(i,j)\in B} E^n_{i,j}
\prod_{(i',j')\in B'} E^n_{i',j'}
\,
\E\!\left[
\prod_{(i,j)\notin\mathbf{i}(B)} X^n_{i,j}
\prod_{(i',j')\notin\mathbf{i}'(B')}
\overline{X^n_{i',j'}}
\right]
\right|
= O\!\left(\frac{1}{n}\right).
\end{align}

If any edge of the multigraph
$\mathbf{C}\cup\mathbf{C}'\setminus(B\cup B')$ has multiplicity one, then the
expectation in \eqref{sum_bound_EtrY^b} vanishes. Otherwise, every edge appears
with multiplicity at least two, and the proof of \eqref{sum_bound_EtrY^b} is
identical to that of \eqref{covariance}.

Consequently,
\[
\E\bigl|\tr((Y^n)^k)\bigr|^2
=
\E\bigl|\tr((X^n)^k)\bigr|^2
+ |\tr((E^n)^k)|^2
+ 2\,\E\bigl[\tr((X^n)^k)\tr((E^n)^k)\bigr]
+ O\!\left(\frac{1}{n}\right).
\]
By Assumption~\ref{ass:E}, $|\tr((E^n)^k)|^2$ is uniformly bounded, and
$\E|\tr((X^n)^k)|^2$ is bounded by Lemma~\ref{lemma_moments_of_X^n}. This proves
the claim.
\end{proof}

\medskip
All the necessary ingredients are now in place to complete the proof of
Theorem~\ref{characteristic_poly_thm}.
 
\begin{proof}[Proof of Theorem~\ref{characteristic_poly_thm}]
Recall the notation from Proposition~\ref{trunc}. We first show that
\begin{align}\label{q^D_sim}
    q_n^D(z)\sim_n b_n(z)\det\!\bigl(I-z X^{n,D}\bigr)
    \sim_n b_n(z)\,\kappa^D(z)\,\exp(-F),
\end{align}
where \(\kappa^D(z)=\sqrt{1-z^2 \mathbb{E}(A^D_{1,1})^2}\).
Moreover, in what follows set
\[
Q_n^D(z)= b_n(z)\,\kappa^D(z)\,\exp(-F),
\qquad
Q_n(z)= b_n(z)\,\kappa(z)\,\exp(-F).
\]

Notice that for \(z\in\mathbb{C}\), the series
\(\sum_{k=1}^\infty \frac{z^k}{k}(Y^{n,D})^k\) is well-defined for \(|z|\) small
enough, and we can express \(q_n^D(z)\) as
\begin{equation}\label{q_=tr}
    q_n^D(z)=\exp\!\left(-\sum_{k=1}^\infty
    \operatorname{tr}\!\bigl((Y^{n,D})^k\bigr)\,\frac{z^k}{k}\right).
\end{equation}

By Proposition~6.1 of \cite{coste2023sparse}, we can rewrite, for \(|z|\) small enough,
\[
\exp\!\left(-\sum_{k=1}^\infty \operatorname{tr}\!\bigl((Y^{n,D})^k\bigr)\frac{z^k}{k}\right)
=
1+\sum_{k=1}^n
P_k\!\Bigl(\operatorname{tr}(Y^{n,D}),\ldots,\operatorname{tr}\!\bigl((Y^{n,D})^k\bigr)\Bigr)\,
\frac{z^k}{k!},
\]
for some polynomials \(P_k\) which do not depend on \(n\).
By analytic continuation,
\[
q_n^D(z)=
1+\sum_{k=1}^n
P_k\!\Bigl(\operatorname{tr}(Y^{n,D}),\ldots,\operatorname{tr}\!\bigl((Y^{n,D})^k\bigr)\Bigr)\,
\frac{z^k}{k!}
\]
for any \(z\in\mathbb{C}\).

Thus, it suffices to examine the joint law of
\(\bigl(\operatorname{tr}(Y^{n,D}),\ldots,\operatorname{tr}\!\bigl((Y^{n,D})^k\bigr)\bigr)\)
for any \(k\in\mathbb{N}\).
In this case, we combine Proposition~\ref{tr(Y)-tr(X)-det},
Lemma~\ref{lemma_moments_of_X^n}, and Lemma~\ref{trace_bounded_Y^n_lem} to conclude
\begin{multline}\label{tracessim}
\bigl(\operatorname{tr}(Y^{n,D}),\ldots,\operatorname{tr}\!\bigl((Y^{n,D})^k\bigr)\bigr)
\sim_n
\bigl(\operatorname{tr}(X^{n,D}),\ldots,\operatorname{tr}\!\bigl((X^{n,D})^k\bigr)\bigr)
+\bigl(\operatorname{tr}(E^n),\ldots,\operatorname{tr}\!\bigl((E^n)^k\bigr)\bigr) \\
\sim_n
\bigl(Z_1+\operatorname{mean}_1^D,\sqrt{2}\,Z_2+\operatorname{mean}_2^D,\ldots,
\sqrt{k}\,Z_k+\operatorname{mean}_k^D\bigr)
+\bigl(\operatorname{tr}(E^n),\ldots,\operatorname{tr}\!\bigl((E^n)^k\bigr)\bigr).
\end{multline}
Notice that the Gaussian random variables \(Z_k\) do not depend on \(D\).
By Proposition~\ref{cvgH} and \eqref{tracessim}, we deduce that \eqref{q^D_sim} holds.

We continue with the proof of \eqref{q_nsimthm}. Fix an integer \(m>0\), and an
\(m\)-tuple \((z_1,\ldots,z_m)\in\mathcal{D}(0,1)^m\).
Let \(\varphi:\mathbb{R}^{2m}\to\mathbb{R}\) be a bounded Lipschitz function.
Since for all \(z\in\mathcal{D}(0,1)\),
\[
\lim_{D\to\infty}\kappa^D(z)=\kappa(z),
\]
it follows that
\[
\sup_n\left|
\mathbb{E}\varphi\bigl(Q_n(z_1),\ldots,Q_n(z_m)\bigr)
-\mathbb{E}\varphi\bigl(Q_n^D(z_1),\ldots,Q_n^D(z_m)\bigr)
\right|\xrightarrow[D\to\infty]{}0.
\]
Therefore,
\begin{align}\label{phiapproxi}
&\left|
\mathbb{E}\varphi\bigl(q_n(z_1),\ldots,q_n(z_m)\bigr)
-\mathbb{E}\varphi\bigl(Q_n(z_1),\ldots,Q_n(z_m)\bigr)
\right| \nonumber\\
&\le
\left|
\mathbb{E}\varphi\bigl(q_n(z_1),\ldots,q_n(z_m)\bigr)
-\mathbb{E}\varphi\bigl(q_n^D(z_1),\ldots,q_n^D(z_m)\bigr)
\right| \nonumber\\
&\quad+
\left|
\mathbb{E}\varphi\bigl(q_n^D(z_1),\ldots,q_n^D(z_m)\bigr)
-\mathbb{E}\varphi\bigl(Q_n^D(z_1),\ldots,Q_n^D(z_m)\bigr)
\right| \nonumber\\
&\quad+
\left|
\mathbb{E}\varphi\bigl(Q_n^D(z_1),\ldots,Q_n^D(z_m)\bigr)
-\mathbb{E}\varphi\bigl(Q_n(z_1),\ldots,Q_n(z_m)\bigr)
\right|.
\end{align}
The first term on the right-hand side is bounded by a positive number
\(\varepsilon_D\) independent of \(n\) and converging to zero as \(D\to\infty\)
by Proposition~\ref{trunc}. The second term converges to zero as \(n\to\infty\)
since \(q_n^D(z)\sim_n Q_n^D(z)\). We just showed that the third term can be
controlled similarly to the first term. Thus, the left-hand side converges to
zero as \(n\to\infty\). By applying Proposition~\ref{cvgH}, we obtain
\(q_n\sim_n Q_n\).

Next we prove \eqref{thm_statement_char_pol}. Set
\[
S_n(z)=b_n(z)\det\!\bigl(I-zX^n\bigr),
\qquad
S_n^D(z)=b_n(z)\det\!\bigl(I-zX^{n,D}\bigr).
\]
Given the bounds from \eqref{phiapproxi} and \eqref{q^D_sim}, it is sufficient to prove that
\[
\sup_n\left|
\mathbb{E}\varphi\bigl(S_n(z_1),\ldots,S_n(z_m)\bigr)
-\mathbb{E}\varphi\bigl(S_n^D(z_1),\ldots,S_n^D(z_m)\bigr)
\right|\xrightarrow[D\to\infty]{}0.
\]
The latter can be proven easily by using Assumption~\ref{ass:E} to bound \(b_n(z)\)
and comparing \(\det(I-zX^n)\) with \(\det(I-zX^{n,D})\) as is done in
Lemma~3.3 of \cite{bordenave2022convergence}.
\end{proof}

\section{Comparison with a Gaussian matrix}
\label{comp-gauss} 
The goal of this section is to compare the spectral properties of matrix $X^n$ as defined in \eqref{sparser_matrices} with analogous quantities of a Gaussian random matrix $G^n\in \R^{n\times n}$ with i.i.d. centered real Gaussian entries, each with variance $n^{-1}$. We mostly rely on results from \cite{brailovskaya2024universality}. 

Let $\xi \in \C$ be such that $|\xi| > 1$. Consider the following matrices
\begin{equation}\label{def:hermit}
    H^n(\xi) = \begin{bmatrix}
        0 & X^n - \xi I_n \\
        (X^n - \xi I_n)^{\star} & 0
    \end{bmatrix}\qquad \textrm{and}\qquad S^n(\xi) = \begin{bmatrix}
        0 & G^n - \xi I_n \\
        (G^n - \xi I_n)^{\star} & 0
    \end{bmatrix}\ . 
\end{equation}
As is well known, the set of eigenvalues of $H^n(\xi)$ counting multiplicities 
coincides with the union of the set of singular values of $X^n-\xi I_n$ counting multiplicities
and the set of the opposites of these singular values. A similar remark holds for $S^n(\xi)$ and $G^n-\xi I_n$. 
For simplicity we will often write $H^n, S^n$ instead of $H^n(\xi), S^n(\xi)$.

We start with a comparison of the empirical spectral distribution of the matrices.

\begin{prop}\label{prop_ESD}
The spectral measures of $H^n$ and $S^n$ are asymptotically equivalent, that 
is, writing 
\[
\nu^{H^n} = \frac{1}{2n} \sum_{i \in [2n]}\delta_{\lambda_i(H^n)} 
 \quad \text{and} \quad 
\nu^{S^n} = \frac{1}{2n} \sum_{i \in [2n]}\delta_{\lambda_i(S^n)} , 
\]
it holds that $\nu^{H^n} \sim \nu^{S^n}$ as random variables valued in the
space of probability measures on $\R$. 
\end{prop}
\begin{proof}
    Follows directly by the discusson in Subsection 10.3 of \cite{rudelson2019sparse}. 
\end{proof}

Next we prove a classical result for sub-gaussian random variables.
\begin{lem}
Let Assumption \ref{ass:subgaussian} hold. Then for some constant $C'>0$,
\begin{align}\label{asymptotiC_Bound}
    \lim_{n \to \infty} \mathbb{P}\left( \max_{i,j} |X^n_{i,j}| \leq C' \left( \frac{\log n}{K_n} \right)^{1/2} \right) = 1. 
\end{align}
\end{lem}
\begin{proof}
By the union bound and Assumption \ref{ass:subgaussian} we get
\begin{eqnarray*}
   \mathbb{P}\left( \max_{i,j} |X^n_{i,j}| > C' \left( \frac{\log n}{K_n} \right)^{1/2} \right)
&\le &\mathbb{P}\left( \max_{i,j} \frac{|A^n_{i,j}|}{\sqrt{K_n}} > C' \left( \frac{\log n}{K_n} \right)^{1/2} \right) \\
&=& \mathbb{P}\left( \max_{i,j}|A^n_{i,j}|> C' \left( \log n \right)^{1/2} \right)\quad 
\le \quad  2 n^{2} \exp( - C (C')^2 \log(n))\,. 
\end{eqnarray*}
It remains to choose $C'$ large enough to conclude.
\end{proof}

For the empirical spectral distribution the finiteness of the second moment of the entries of $A^n$ was sufficient. For finer results one needs to make more assumptions for the matrix $A^n$. Recall that $\C^+ = \{ z \in \C \, : \, 
 \Im z > 0 \}$ and that $s_n(M)$ denotes the least singular value of any $n\times n$ matrix $M$.

We now present our comparison result, a corollary of \cite{brailovskaya2024universality}. 

\begin{thm}\label{comparison_with_gaussian_thm} Let Assumptions \ref{ass:K} and \ref{ass:subgaussian} hold. Let $G^n$ be a $n\times n$ matrix with i.i.d. centered real Gaussian entries each with variance $n^{-1}$ and the matrices $H^n(\xi)$ and $S^n(\xi)$ be defined by \eqref{def:hermit}. Let $\xi\in \mathbb{C}$ with $|\xi|>1$. Assume that 
$$
        \lim_{n \to \infty} \frac{\log ^9(n)}{K_n} = 0\,,
$$
Then
\begin{enumerate}[label=(\alph*)]
    \item\label{snX-G} for every $\varepsilon>0$, 
$\quad {\displaystyle 
        \lim_{n \to \infty} \mathbb{P}\left( \left| s_n(X^n - \xi I) - s_n(G^n - \xi I) \right| \geq \varepsilon \right) = 0}\, ,
$

\vspace{0.1cm}

    \item\label{E(X-G)} for every $z \in \C^+$,
    $\quad {\displaystyle 
        \lim_{n \to \infty} \left\| \E(S^n - zI)^{-1} - \E(H^n - zI)^{-1} \right\| = 0}\, .
    $
\end{enumerate}
\end{thm}

\begin{rem}
    In the previous theorem, it turns out that assumption $\log^9(n)/K_n\to 0$ is required to prove item \ref{snX-G}. The lighter assumption $\log(n)/K_n\to 0$ is sufficient to establish item \ref{E(X-G)}, see the proof below. 
\end{rem}

\begin{proof} 
As a consequence of Assumption~\ref{ass:subgaussian}, the following estimate holds 
\begin{equation}\label{est:subgaussian}    
\E \max_{i,j \in [n]} |A^n_{ij}|^2 \leq C_1 \log n\, ,
\end{equation}
for some constant $C > 0$ (see, for instance, \cite[Exercises 2.26 and 2.44]{vershynin2018high}). 
Moreover, the singular value $s_n(G^n - \xi I)$ is positive with probability one and coincides
on this probability one set with the smallest positive eigenvalue of $S_n$.

In order to establish \ref{snX-G} we shall rely on \cite[Theorem 2.8]{brailovskaya2024universality}. Notice first that 
$$
|s_n(X^n-\xi I) - s_n(G^n-\xi I)| \le d_{\boldsymbol{H}}(\sigma(H^n), \sigma(S^n))\ ,
$$
where $\sigma(H^n)$ and $\sigma(S^n)$ are respectively the spectra of $H^n$ and $S^n$. The following quantities whose estimates are straightforward appear in the statement of \cite[Theorem 2.8]{brailovskaya2024universality}:
\begin{eqnarray*}
    \kappa&=& \left\| \mathbb{E} (H - \mathbb{E} H)^2\right\|^2 \ =\ 1\ ,\\
    \kappa_*&=& \sup_{\|v\|, \|w\|=1} \left(\mathbb{E} \left| \langle v, (H-\mathbb{E} H)w\rangle \right|^2 \right)^{1/2}\ \le \ \frac 2{\sqrt{n}} \ ,\\
    \overline{R} &=& \left( \mathbb{E} \max_{ij} |X_{ij}^n|^2 \right)^{1/2}\ \le \ \left( \frac {C_1\log(n)}{K_n}\right)^{1/2}\, .
\end{eqnarray*}
Now the theorem states that:
$$
\mathbb{P} \left\{ \left| s_n(G^n- \xi I) - s_n(X_n- \xi I)\right| \ge C_0\varepsilon (t)\,;\ \max_{i,j} |X_{ij}|\le R \right\} \le 2n \exp\left(-t\right)\ ,
$$ 
for every $t\ge 0$ with the conditions:
$$
(i)\quad R\ge \sqrt{\kappa \,\overline{R}} + \sqrt{2}\, \overline{R}\qquad \textrm{and}\qquad (ii) \quad 
\varepsilon_R(t) = \kappa_* \sqrt{t} + R^{1/3} \kappa^{2/3} t^{2/3} + Rt\, , 
$$
and where $C_0$ is a universal constant. 

We first set $R=C_2\left( \frac{\log(n)}{K_n}\right)^{1/4}$ and notice that for this value, $\mathbb{P}\left\{ \max_{i,j}|X_{ij}|\le R\right\}\to_n 1$. In fact, using estimate \eqref{est:subgaussian} we have
$$
\mathbb{P}\left\{ \max_{i,j}|X_{ij}|>  R\right\} \ \le \ \frac{C_1}{C_2^2}\sqrt{\frac{\log(n)}{K_n}}\quad \xrightarrow[n\to\infty]{}\quad 0\ ,
$$
by assumption. Now setting 
$
t= C_3 \left( \frac{K_n}{\log(n)}\right)^{1/8}
$, we get 
$$
\varepsilon_R(t) = C_2^{1/3} C_3^{2/3} +o(1)\, .
$$
Notice that with such a choice, 
$$
2n e^{-t} = 2\exp{\small \left\{ \log(n) - C_3\left(\frac{K_n}{\log(n)}\right)^{1/8}\right\}} = 
2\exp{\small \left\{ \log(n) \left[ 1 - C_3\left(\frac{K_n}{\log^9(n)}\right)^{1/8}\ \right]\right\}}\
\xrightarrow[n\to\infty]{}\ 0
$$
by the condition $K_n/\log^9(n)\to \infty$.
It remains to choose $C_3$ so that $C_2^{1/3}C_3^{2/3} =\varepsilon$ to conclude that 
\begin{multline*}
\mathbb{P} \left\{ \left| s_n(G^n- \xi I) - s_n(X_n- \xi I)\right| \ge 2C_0\varepsilon \right\}\\
\le \mathbb{P} \left\{ \left| s_n(G^n- \xi I) - s_n(X_n- \xi I)\right| \ge C_0\varepsilon (t)\,;\ \max_{i,j} |X_{ij}|\le R \right\} + 
\mathbb{P} \left\{ \max_{i,j} |X_{ij}|> R \right\}\quad \xrightarrow[n\to\infty]{}\quad 0\,.
\end{multline*}

In order to establish \ref{E(X-G)} we shall rely on \cite[Theorem 2.11]{brailovskaya2024universality} which yields that for every $z\in \mathbb{C}^+$
$$
\left\|  
\mathbb{E}(zI - H)^{-1} - \mathbb{E}(zI-S)^{-1}
\right\|\ \le \ \frac{\kappa^* +\overline{R}^{1/10}}{\Im^2(z)} \ = \ \frac 1{\Im^2(z)} \left( \frac 2{\sqrt{n}} + \left( \frac{\log(n)}{K_n}\right)^{1/20}\right)\, .
$$
Proof of Theorem \ref{comparison_with_gaussian_thm} is completed.

\end{proof}

\begin{rem}
    The results of \cite{brailovskaya2024universality} are fairly general. One may relax the sub-Gaussian assumption (Assumption \ref{ass:subgaussian}) at the cost of increasing the sparsity parameter $K_n$ and still have an analogue of Theorem \ref{comparison_with_gaussian_thm}.
\end{rem}

We now prove a concentration result. Recall that $X^n$'s entries write $X^n_{ij} =\frac {B_{ij}^n A_{ij}^n}{\sqrt{K_n}}$.

\begin{lem}\label{lem_conc_bilinear}
Assume that $\mathbb{E}|A_{11}^n|^8<\infty$. Let $z \in \C^+$, $\varepsilon>0$ and 
consider two sequences $(\tilde w^{2n})$ and $(\tilde q^{2n})$ of unit vectors in $\C^{2n}$, where 
$$
\tilde w^{2n}_i=\tilde q^{2n}_i=0\quad \textrm{for}\quad i\in \{n+1,\cdots, 2n\}\, .
$$
Then
$$
    \lim_{n \to \infty} \mathbb{P} \left( \left| \langle (H^n(\xi) - zI)^{-1}\tilde w^{2n},\tilde q^{2n}\rangle - \E  \langle (H^n(\xi) - zI)^{-1}\tilde w^{2n},\tilde q^{2n}\rangle \right| \geq \epsilon \right) = 0\ .
$$
\end{lem}

\begin{rem} In the proof below, the condition $\mathbb{E}|A_{11}|^4<\infty$ appears in estimating the variance of a quadratic form, see for instance \eqref{eq:fact-2}. The eight moment is required when relying on \cite[Theorem 3.6]{HLNV08}. 
\end{rem}

\begin{proof} We write 
$$\tilde w^{2n}=\begin{pmatrix}
    w^n\\0_n
\end{pmatrix}\quad \textrm{and}\quad \tilde q^{2n}=\begin{pmatrix}
    q^n\\0_n
\end{pmatrix}\ ,
$$ where $w^n,q^n\in \mathbb{C}^n$ and $0_n$ is the null vector in $\mathbb{C}^n$. We will soon drop the index $n$ and simply write $w,q,I$ instead of $w^n,q^n,I_n$. In the sequel, $C$ denotes a constant whose value may change from line to line.

By the Schur complement formula, we have
$$
\langle (H^n(\xi) - zI_{2n})^{-1}\tilde w^{2n},\tilde q^{2n}\rangle = z\,\langle w^n ,\left( -z^2 I_n + (X-\xi I_n)(X-\xi I_n)^\star\right)^{-1} q^n\rangle\ ,
$$
and we are led to study the concentration of the quadratic form $\langle w, Q\, q\rangle$ with 
$$
Q= z\left( -z^2 I + (X-\xi I)(X-\xi I)^\star\right)^{-1}\, .
$$
Notice that $Q$ being the top-left corner of matrix $(H^n-zI)^{-1}$, we immediately get $\| Q\| \le (\Im(z))^{-1}$. Denote by 
$$
Y=X-\xi I 
$$
and let the $(y_i)$'s being the columns of matrix $Y$. In particular, $y_i=x_i- \xi e_i$ and
$$
Q=z\left( -z^2 + Y Y^\star\right)^{-1} = z \left( -z^2 + \sum_{k=1}^n y_k y_k^\star\right)^{-1} \, .
$$
For further use, we introduce $Q^i= z \left( -z^2 + \sum_{k\neq i} y_k y_k^\star\right)^{-1} $.
Denote by 
$$
f(y_1,\cdots, y_n) = \langle w, Q\, q\rangle \, .
$$
Let $\check f_i$ be the function $f$ evaluated at $(y_1,\cdots, y_{i-1}\,,\, {\check y_i}\,,\,y_{i+1},\cdots, y_n)$ where $\check y_i$ is an independent copy of $y_i$.
By Efron-Stein's inequality \cite[Theorem 3.1]{boucheron2003concentration} we have
$$
\mathrm{var}(f) \le \frac 12 \sum_{i=1}^n \mathbb{E} |f - \check f_i|^2\, .
$$
We will rely on the following elementary facts. Let $M\in \mathbb{C}^{n\times n}$ a deterministic matrix, then 
\begin{eqnarray}
\mathbb{E} (y_i^\star M y_i) &=& \frac 1n \mathrm{Trace} (M) + |\xi|^2 M_{ii} \label{eq:fact-1}\ ,\\
\mathrm{var}(y_i^\star M y_i) &\le & C\left( \frac{\mathbb{E}|A_{11}|^4}{nK_n}\mathrm{Trace}(M M^\star) + \frac{|\xi|^2 (M M^\star)_{ii}}n\right)\ .
\label{eq:fact-2}
\end{eqnarray}
The function $z\mapsto y_i^\star Q^i(z) y_i$ is the Stieltjes transform of a non-negative measure, and the function 
$$
z\quad \mapsto\quad  - \frac 1{z+ y_i^\star Q^i(z) y_i}\, . 
$$
is the Stieltjes transform of a probability measure. 
In particular 
$$
\left| 
\frac 1{z+ y_i^\star Q^i(z) y_i}
\right| \quad \le \quad \frac 1{\Im(z)}\, .
$$
In the sequel, we denote by $\mathbb{E}_i=\mathbb{E}\left( \ \cdot\ \mid y_k,\, k\neq i\right)$ and by $\mathrm{var}_i$ the associated conditional variance. 
Using Sherman-Morrisson's inequality, we get
\begin{eqnarray*}
    \mathbb{E} \left| f - \check f_i\right|^2 &= &  \mathbb{E} \mathbb{E}_i \left| 
    \frac{y_i^\star Q^i q w^\star Q^i y_i}{z + y_i^\star Q^i y_i} - \frac{\ {\check y_i}^\star Q^i q w^\star Q^i \check y_i }{z + {\check y}_i^\star Q^i y_i} 
    \right|^2 \ ,\nonumber\\
     &\stackrel{(a)}\le  &  2 \mathbb{E} \mathbb{E}_i \left| 
    \frac{y_i^\star Q^i q w^\star Q^i y_i}{z + y_i^\star Q^i y_i} - \frac{\mathbb{E}_i\left( y_i^\star Q^i q w^\star Q^i y_i\right) }{z + \mathbb{E}_i (y_i^\star Q^i y_i)} 
    \right|^2 + 2 \mathbb{E} \mathbb{E}_i \left| 
    \frac{{\check y_i}^\star Q^i q w^\star Q^i \check y_i}{z + \check y_i^\star Q^i \check y_i} - \frac{\mathbb{E}_i\left( \check y_i^\star Q^i q w^\star Q^i \check y_i\right) }{z + \mathbb{E}_i (\check y_i^\star Q^i \check y_i)} 
    \right|^2 \ ,\nonumber\\
&= & 4 \mathbb{E} \mathbb{E}_i \left| 
    \frac{y_i^\star Q^i q w^\star Q^i y_i}{z + y_i^\star Q^i y_i} - \frac{\mathbb{E}_i\left( y_i^\star Q^i q w^\star Q^i y_i\right) }{z + \mathbb{E}_i (y_i^\star Q^i y_i)} 
    \right|^2\ ,
    \end{eqnarray*}
where $(a)$ follows from the introduction of the auxiliary term 
$$
\frac{\mathbb{E}_i\left( y_i^\star Q^i q w^\star Q^i y_i\right) }{z + \mathbb{E}_i (y_i^\star Q^i y_i)}
= \frac{\mathbb{E}_i\left( \check y_i^\star Q^i q w^\star Q^i \check y_i\right) }{z + \mathbb{E}_i (\check y_i^\star Q^i \check y_i)} \ ,
$$
and the elementary inequality $|a+b|^2\le 2|a|^2+ 2|b|^2$. Introducing appropriate auxiliary terms and proceeding similarly, we get
\begin{eqnarray}
    \mathbb{E} \left| f - \check f_i\right|^2 &\le & 8 \mathbb{E} \mathbb{E}_i \left| 
    \frac{y_i^\star Q^i q w^\star Q^i y_i}{z + y_i^\star Q^i y_i} - \frac{\mathbb{E}_i\left( y_i^\star Q^i q w^\star Q^i y_i\right) }{z + y_i^\star Q^i y_i}
    \right|^2 \nonumber \\
    &&\quad + 8 \mathbb{E} \mathbb{E}_i \left| \mathbb{E}_i\left( y_i^\star Q^i q w^\star Q^i y_i\right) \left\{ 
    \frac{1}{z + y_i^\star Q^i y_i} - \frac{1}{z + \mathbb{E}_i (y_i^\star Q^i y_i)} \right\}
    \right|^2 \ ,\nonumber \\
&\le& \frac{8}{\Im^2(z)} \mathbb{E} \mathrm{var}_i (y_i^\star Q^i q w^\star Q^i y_i) + \frac 8{\Im^4(z)} \mathbb{E} \left| 
\mathbb{E}_i(y_i^\star Q^i q w^\star Q^i y_i) \left( y_i^\star Q^i y_i - \mathbb{E}_i (y_i^\star Q^i y_i)\right)
\right|^2 \, ,\nonumber \\
&=& {\mathcal O}_z \left( \mathbb{E} \mathrm{var}_i (y_i^\star Q^i q w^\star Q^i y_i)\right) + {\mathcal O}_z \left( \mathbb{E} \left| 
\mathbb{E}_i(y_i^\star Q^i q w^\star Q^i y_i) \left( y_i^\star Q^i y_i - \mathbb{E}_i (y_i^\star Q^i y_i)\right)
\right|^2\right)\, .
\label{eq:variance-estimate}
    \end{eqnarray}
We first estimate $\mathrm{var}_i (y_i^\star Q^i q w\, Q^i y_i)$. By \eqref{eq:fact-2} we have
\begin{eqnarray}
\mathrm{var}_i ( y_i^\star Q^i q\, w^\star Q^i y_i) &\le& C \left( \frac{\mathrm{Trace}( Q^i q\, w^\star Q^i [ Q^i]^\star w\, q^\star [Q^i]^\star)}{nK_n} 
+ \frac{|\xi|^2 \left( Q^i q\, w^\star Q^i [ Q^i]^\star w\, q^\star [Q^i]^\star \right)_{ii}}{n}\right)\ ,\nonumber \\
&=& {\mathcal O}_z \left( 
\frac{1}{nK_n}\right) 
+ {\mathcal O}_{z,\xi} \left( \frac{\left( Q^i q\, q^\star [Q^i]^\star \right)_{ii}}{n}
\right)\, . \label{eq:estimate-variance-1}
\end{eqnarray}
We now estimate $\mathbb{E}_i(y_i^\star Q^i q w^\star Q^i y_i)$. By \eqref{eq:fact-1} we have
\begin{eqnarray*}
 \mathbb{E}_i(y_i^\star Q^i q w^\star Q^i y_i) &=&   \frac 1n \mathrm{Trace} (Q^i q w^\star Q^i) + |\xi|^2 \left( Q^i q w^\star Q^i\right)_{ii}
 ={\mathcal O}_z \left( \frac 1n \right) + |\xi|^2 \left( Q^i q w^\star Q^i\right)_{ii}\, .
\end{eqnarray*}
Notice that 
\begin{eqnarray*}
\left| \left( Q^i q w^\star Q^i\right)_{ii}\right|^2 &=&  \left( Q^i q w^\star Q^i\right)_{ii}\times \left( [Q^i]^\star w q^\star [Q^i]^\star\right)_{ii}\\
&\le& \ \left( Q^i q w^\star Q^i [Q^i]^\star w q^\star [Q^i]^\star\right)_{ii}\quad \le\quad  \frac{1}{\Im^2(z)} \left( Q^i q\,q^\star [Q^i]^\star\right)_{ii}\, .
\end{eqnarray*}
Hence 
\begin{equation}
\left| \mathbb{E}_i(y_i^\star Q^i q w^\star Q^i y_i)\right|^2  = {\mathcal O}_z\left( \frac 1{n^2}\right) + {\mathcal O}_{z,\xi}\left( \left| \left( Q^i q w^\star Q^i\right)_{ii}\right|^2\right)= {\mathcal O}_{z}\left( \frac 1{n^2}\right) + {\mathcal O}_{z,\xi}\left( 
 \left( Q^i q\,q^\star [Q^i]^\star\right)_{ii} \right)\ .
\label{eq:estimate-variance-2}
\end{equation}

We finally estimate $\mathrm{var}_i (y_i^\star Q^i y_i)$. By \eqref{eq:fact-2} we have
\begin{equation}
\mathrm{var}_i ( y_i^\star Q^i y_i) \le  \widetilde K\left( \frac{\mathrm{Trace}( Q^i  [ Q^i]^\star )}{nK_n} 
+ \frac{|\xi|^2 \left[ Q^i [Q^i]^\star \right]_{ii}}{n}\right) = {\mathcal O}_z \left( \frac 1{K_n}\right) +{\mathcal O}_{z,\xi}\left( 
\frac{1}{n}\right) =
 {\mathcal O}_{z,\xi}\left( \frac{1}{K_n}\right)\, .
\label{eq:estimate-variance-3}
\end{equation}
Notice that the final upper estimate of $\mathrm{var}_i ( y_i^\star Q^i y_i)$ above is deterministic. Noticing that 
$$\mathbb{E}_i \left\{
|\mathbb{E}_i U|^2\times |V|^2\right\}= |\mathbb{E}_i U|^2 \mathbb{E}_i |V|^2\ ,
$$
and using \eqref{eq:estimate-variance-2}-\eqref{eq:estimate-variance-3}, we get
\begin{eqnarray}
    \mathbb{E}\mathbb{E}_i\left| 
\mathbb{E}_i(y_i^\star Q^i q w^\star Q^i y_i) \left( y_i^\star Q^i y_i - \mathbb{E}_i (y_i^\star Q^i y_i)\right)
\right|^2
&=&
\mathbb{E} \left\{ \left| 
\mathbb{E}_i(y_i^\star Q^i q w^\star Q^i y_i)\right|^2 \mathbb{E}_i\left| y_i^\star Q^i y_i - \mathbb{E}_i (y_i^\star Q^i y_i)
\right|^2\right\}\nonumber\,,\\
&=& \mathbb{E} \left\{ \left| 
\mathbb{E}_i(y_i^\star Q^i q w^\star Q^i y_i)\right|^2 \mathrm{var}_i ( y_i^\star Q^i y_i)\right\}\,,\nonumber \\
&=& {\mathcal O}_{z,\xi} \left( \frac{1}{n^2K_n}\right) + {\mathcal O}_{z,\xi} \left( \frac{\mathbb{E}\left( Q^i q\,q^\star [Q^i]^\star\right)_{ii}}{K_n}\right) \, .
\label{eq:estimate-variance-4}
\end{eqnarray}
Plugging back estimates \eqref{eq:estimate-variance-1} and \eqref{eq:estimate-variance-4} into \eqref{eq:variance-estimate} and summing over $i$ finally yields
$$
    \mathrm{var} (f) \ \le \ \frac 12 \sum_i \mathbb{E}|f- \check f_i|^2\ =\  {\mathcal O}_{\xi,z} \left( \frac 1{K_n} +\frac{\sum_i \mathbb{E}(Q^i q\, q^\star [Q^i]^\star)_{ii}}{K_n}\right)\, .
$$
It remains to notice that 
$$
\sum_i \mathbb{E}\left( Q^i q\, q^\star [Q^i]^\star \right)_{ii} ={\mathcal O}_z(1)
$$
by \cite[Theorem 3.6]{HLNV08} to conclude.

\end{proof}

We now present fairly standard results concerning the Gaussian matrix $S^n$.
\begin{thm}\label{thm_gaussians_properties}
    Let $\xi\in \mathbb{C}$ with $|\xi|>1$. Let $G^n$ be a $n\times n$ matrix with i.i.d. real Gaussian entries each with variance $n^{-1}$ and $S^n(\xi)$ the $2n\times 2n$ matrix defined by \eqref{def:hermit}.
    
    The following facts hold true for matrix $S^n(\xi)$.
    \begin{enumerate}[label=(\alph*)]
        \item\label{thm_gaussians_e.s.d.} There exists a probability measure $\mu^{\xi}$ such that 
        \begin{align*}
      \frac{1}{2n} \sum_{i \in [2n]}\delta_{\lambda_i(S^n)} 
\quad \underset{n \to \infty}{\Longrightarrow}\quad \mu^{\xi} \qquad  \text{a.s.}
        \end{align*}

        \item\label{thm_gaussians_support_density} The probability measure $\mu^{\xi}$ is symmetric and has a density supported in $(-C_\xi,-c_\xi)\cup(c_\xi,C_\xi)$ for some positive constants $0<c_\xi<C_\xi$. 
        \item\label{thm_gaussians_fixed_point_stieltjes}  If $m^{\xi}$ denotes the Stieltjes transform of $\mu^{\xi}$, then $m^{\xi}$ is the unique function that satisfies the following fixed point equation 
\begin{align}
	\label{eq:scmde}
	-\frac{1}{{m}^{\xi}(w)}=w+{m}^{\xi}(w)-\frac{|\xi|^2}{w+{m}^\xi(w)}, \quad \mbox{with}\quad \Im(m^{\xi}(w))\,,\ \Im(w)>0,
\end{align}
\item\label{thm_gaussian_least_singular_value} There exists a positive constant $\check c_\xi$ such that 
\begin{align*}
 \lim_{n \to \infty}\mathbb{P}(s_n(G^n-\xi I)\geq \check c_\xi)\ =\ 1\,.
\end{align*}
\end{enumerate}
Let $\tilde w^{2n}, \tilde q^{2n}$ be two deterministic unit vectors in $\C^{2n}$ satisfying
 $
 \tilde w^{2n}_i=\tilde q_i^{2n} =0$ for $i\ge n+1$.
\begin{enumerate}[label=(\alph*),resume]
\item\label{thm_bilinear_gaussians} Let $\eta \in \R^+$, then
\begin{align*}
   \lim_{n \to \infty} \left|\left\langle \tilde w^{2n},(S^n-i\eta I)^{-1}\tilde q^{2n}\right\rangle  -  m^\xi (i\eta) \langle \tilde w^{2n},\tilde q^{2n}\rangle\right|=0 \quad \text{a.s.}\,.
\end{align*}
 \item \label{item:concentration} Let $z \in \C^+$, then for every $\varepsilon>0$, 
    $$
        \lim_{n \to \infty} \mathbb{P}\left( \left| \left\langle\tilde w^{2n},(S^n - zI)^{-1}\tilde q^{2n} \right\rangle - \left\langle \tilde w^{2n},\E(S^n - zI)^{-1}\tilde q^{2n} \right\rangle \right| \geq \varepsilon \right) = 0\,.
    $$
    \end{enumerate}
\end{thm}
\begin{proof}
Random matrix models like $S^n$ are very popular and have been heavily studied.
\ref{thm_gaussians_e.s.d.} and \ref{thm_gaussians_support_density} can be found in Proposition 3.1 of \cite{bourgade2014local};\ref{thm_gaussians_fixed_point_stieltjes} can be found in \cite[(2.17)]{cipolloni2023rightmost}; \ref{thm_gaussian_least_singular_value} can be proven by a direct application of \cite[Theorem 1.1]{dozier2007analysis}. Finally \ref{thm_bilinear_gaussians} and \ref{item:concentration} are consequences of \cite[Theorem 1.1]{HLNV08}.
\end{proof}
\begin{cor}\label{cor_for_im(eta/eta}
    Let $\eta>0$ and $m^{\xi}$ the Stieltjes transform defined in Theorem \ref{thm_gaussians_properties}-\ref{thm_gaussians_fixed_point_stieltjes}, then one has:
    \begin{align*}
        \lim_{\eta \to 0} \frac{\Im (m^{\xi}(i \eta))}{\eta}=\frac{1}{|\xi|^2-1}
    \end{align*}
\end{cor}
\begin{proof} Let $\rho^{\xi}$ denote the density of $\mu^{\xi}$, notice that $\rho^\xi$ is symmetric. Recall that $\mu^{\xi}$ is supported in $(-C_\xi,-c_\xi)\cup(c_\xi,C_\xi)$ for positive constants $c_\xi, C_\xi$.
   First, define the function $$h(\eta):=\frac{\Im(m^{\xi}(i\eta))}{\eta}=2\int_{c_\xi}^{C_{\xi}} \frac{\rho^{\xi}(x)}{x^2+\eta^2} dx\, .
   $$ 
  Then $h(\eta)$ is Lipschitz continuous on a small interval $(0,\varepsilon)$ with $\varepsilon<c_\xi$ since
\begin{align*}
    |h(\eta_1)-h(\eta_2)|\quad \leq\quad  4\left| \eta_1-\eta_2 \right| \,\varepsilon \int_{c_\xi}^{C_\xi} \frac{\rho^{\xi}(x)}{(x^2+\eta_1^2)(x^{2}+\eta_2^{2})}dx\quad \leq\quad \frac{4 \varepsilon}{c_\xi^4} \left| \eta_1-\eta_2 \right|  \, .
\end{align*}
In particular, the limit
    $\lim_{\eta \to 0}h(\eta)$
exists. The symmetry of the density $\rho^\xi$ yields that $\overline{m^\xi(i\eta)}= - m^\xi(i\eta)$ hence 
$$\Re\,m^\xi(i\eta)=0\,.$$ 
Rewriting the fixed point equation in Theorem \ref{thm_gaussians_properties}-\ref{thm_gaussians_fixed_point_stieltjes} in terms of function $h(\eta)$ yields
\begin{align}\label{ineq_h}
    1= h(\eta) \eta^2(1+h(\eta))+\frac{|\xi|^2 h(\eta)}{1+h(\eta)}\, .
\end{align}
Taking the limit of \eqref{ineq_h} as $\eta\to0$ we end up with the desired result:
$$
  h(0)=\frac{1}{|\xi|^2-1}\, .
$$
\end{proof}

We are now in position to compare quadratic forms based on the resolvent of $H^n$ and on the resolvent of $S^n$.

\begin{cor}\label{bilinear_comparison_cor}
Let $A^n$ satisfy Assumption \eqref{ass:subgaussian}, $z\in \mathbb{C}^+$, $\varepsilon>0$ and
\begin{align*}
    \lim_{n\to \infty} \frac{\log n}{K_n} = 0\, .
\end{align*}
Let $\tilde w^{2n}, \tilde q^{2n} \in \C^{2n}$ be deterministic unit vectors satisfying $\tilde w^{2n}_i=\tilde q^{2n}_i =0$ for $i\ge n+1$, 
then
\begin{align*}
    \lim_{n \to \infty} \mathbb{P} \left( \left| \langle (H^n(\xi) - zI)^{-1}\tilde w^{2n},\tilde q^{2n}\rangle -  \langle (S^n(\xi) - zI)^{-1}\tilde w^{2n}\,,\,\tilde q^{2n}\rangle \right| \geq \epsilon \right) = 0.
\end{align*}
\end{cor}

\begin{proof}
In the notations below, we drop the indices. The claim follows from the inequality
\begin{eqnarray*}
\lefteqn{    \left| \langle (H - zI)^{-1}\tilde w,\tilde q\rangle -  \langle (S - zI)^{-1}\tilde w,\tilde q\rangle \right| }\\
    &\leq& \left| \langle (H - zI)^{-1}\tilde w,\tilde q\rangle - \E \langle (H - zI)^{-1}\tilde w,\tilde q\rangle \right| \\
    && + \left| \langle (S - zI)^{-1}\tilde w,\tilde q\rangle - \E \langle (S - zI)^{-1}\tilde w,\tilde q\rangle \right| + \left\| \E(S - zI)^{-1} - \E(H - zI)^{-1} \right\|\,.
\end{eqnarray*}
The first term of the r.h.s. goes to zero in probability by Lemma \ref{lem_conc_bilinear}; the second term goes to zero by Theorem \ref{thm_gaussians_properties}-\ref{item:concentration}; the last term goes to zero by Theorem \ref{comparison_with_gaussian_thm}-\ref{E(X-G)}.

\end{proof}

\section{Proof of Theorem \ref{thm_eigenvector}}
\label{section_eigenvectors}
Recall the definition of $X^n$ in \eqref{sparser_matrices} and the fact that $Y^n=X^n+u^n(v^n)^\star$.
In all this section, we shall assume without generality loss that 
\[
 \langle v^n,u^n\rangle \ \xrightarrow[n\to\infty]{} \ \xi \in \C \qquad 
 \text{with} \quad |\xi| > 1, 
\]
since it is sufficient to establish the convergence in probability to all 
sub-sequential limits of $\langle v^n,u^n\rangle$.  

We start our analysis with a well-known linear algebra result (see, 
\emph{e.g.}, \cite{benaych2011eigenvalues,tao2013outliers}) that we prove for
completeness. 
\begin{lem}
\label{benaych} 
Let $z_0 \not\in\sigma(X^n)$. Then, $z_0\in\sigma(Y^n)$ if and only if
\[
1+\left\langle (X^n-z_0I)^{-1} u^n, v^n \right\rangle=0.
\] 
The case being, a right eigenvector corresponding to the eigenvalue $z_0$ of 
$Y^n$ is 
\[
   (X^n-z_0I)^{-1}u^n.
\] 
\end{lem}
\begin{proof}

For the first part, since $z_0$ is not an eigenvalue of $X^n$ and by the
property that $\operatorname{det}(I+AB)=\operatorname{det}(I+BA)$ for 
rectangular matrices $A$ and $B$ with compatible dimensions, we get that 
\[
    \frac{\operatorname{det}(Y^n-z_0I)}{\operatorname{det}(X^n-z_0I)}= 
 \operatorname{det}(I+(X^n-z_0I)^{-1}u^n(v^n)^\star) 
 = 1+ \left\langle (X^n-z_0I)^{-1}u^n, v^n \right\rangle . 
\]
The claim follows. 

For the second part, for $z_0 \not\in \sigma(X^n)$, we have that
\[
    (Y^n-z_0I)(X^n-z_0I)^{-1}u^n = u^n+u^n (v^n)^\star(X^n-z_0I)^{-1}u^n 
    =\left(\left\langle (X^n-z_0I)^{-1}u^n,v^n\right\rangle+1\right) u^n.
\] 
Due to the first part of the lemma, if $z_0$ is an eigenvalue of $Y^n$, the 
right hand side of this expression is zero. Thus 
$Y^n(X^n-z_0I)^{-1}u^n = z_0 (X^n-z_0I)^{-1}u^n$ which is the required result. 
\end{proof}

Let us briefly present the strategy of proof. Thanks to the former result, we are 
led to study the behavior of 
\[
\left\langle \frac{u^n}{\| u^n \|},  
  \frac{(X^n - \lambda_{\max}(Y^n)I)^{-1} u^n}
  {\| (X^n - \lambda_{\max}(Y^n)I)^{-1} u^n \|} \right\rangle 
\]
on an appropriate probability event.  With the help of the results of the
former section, we first show that $(X^n-\lambda_{\max}(Y^n)I)^{-1}$ can be
replaced with $(X^n-\xi I)^{-1}$ in this expression. This is the aim of
Lemma~\ref{l-mu} below.  With the help of
Theorem~\ref{characteristic_poly_thm}, we then consider the asymptotics of
$\langle u^n, (X^n-\xi I)^{-1} u^n\rangle / \| u^n \|^2$ (Lemma~\ref{lm:inp}).
The remainder of the proof consists in studying $\| (X^n-\xi I)^{-1} u^n \| /
\| u^n \|$ with help of the results of Section~\ref{comp-gauss} again.


\begin{lem}
\label{l-mu} 
There exists a sequence $(\mathcal E_n^{\ref{l-mu}})$ of probability events 
such that $1_{\mathcal E_n^{\ref{l-mu}}} \to 1$ in probability, 
the smallest singular values of $X^n-\xi I$ and $X^n-\lambda_{\max}(Y^n)I$ are 
lower bounded by positive constants on $\mathcal E_n^{\ref{l-mu}}$, and 
moreover, it holds that 
\[
 1_{\mathcal{E}_n^{\ref{l-mu}}} 
 \|(X^n-\lambda_{\max}(Y^n)I)^{-1}-(X^n-\xi I)^{-1} \|  
 \quad
  \xrightarrow[n\to\infty]{\mathbb{P}} \quad 0\,.
\]
\end{lem} 
\begin{proof}
We mainly need to control the smallest singular value $s_n(X^n-\xi I)$, and to
use Corollary~\ref{lmax-uv}, which shows in our context that
$$
\lambda_{\max}(Y^n) \xrightarrow[n\to\infty]{\mathbb{P}} \xi\, .
$$

To control $s_n(X^n-\xi I)$, we apply Theorems
\ref{comparison_with_gaussian_thm}-\ref{snX-G} and
\ref{thm_gaussians_properties}-\ref{thm_gaussian_least_singular_value} to
obtain the existence of a constant $c>0$ satisfying
\[
 \lim_n \mathbb{P}\left\{s_n(X^n-\xi I)\geq c\right\}=1\,.
\]
Defining the event 
\[
\mathcal E_n^{\ref{l-mu}} = \left\{ s_n(X^n-\xi I) \geq c \right\} \cap 
          \left\{ \left| \lambda_{\max}(Y^n) - \xi \right| \leq c/2\right\} , 
\]
we know from what precedes that $\mathbb P\{ \mathcal E_n^{\ref{l-mu}} \} \to_n 1$. Moreover,
by Weyl's inequality, we obtain that  
\[
s_n(X^n-\lambda_{\max}(Y^n) I) \quad \geq\quad  s_n(X^n-\xi I) - 
   |\xi- \lambda_{\max}(Y^n)|\,.
\]
Therefore, $s_n(X^n-\lambda_{\max}(Y^n) I) \geq c/2$ on 
$\mathcal E_n^{\ref{l-mu}}$, and both matrices $X^n-\xi I$ and 
$X^n-\lambda_{\max}(Y^n) I$ have their smallest singular values lower bounded
by a positive constant 
on $\mathcal E_n^{\ref{l-mu}}$. In particular, the expression 
$$1_{\mathcal{E}_n^{\ref{l-mu}}}
 \|(X^n-\lambda_{\max}(Y^n)I)^{-1}-(X^n-\xi I)^{-1} \|
 $$ is well-defined. On 
$\mathcal E_n^{\ref{l-mu}}$, we furthermore have 
\begin{eqnarray*} 
 \|(X^n-\lambda_{\max}(Y^n)I)^{-1}-(X^n-\xi I)^{-1} \| &= &
 \| (X^n-\lambda_{\max}(Y^n)I)^{-1}(X^n-\xi I)^{-1} 
    (\lambda_{\max}(Y^n) - \xi) \|\,, \\
 &\leq& | \lambda_{\max}(Y^n) - \xi | \ 
  \|(X^n-\lambda_{\max}(Y^n)I)^{-1} \| \ \|(X^n-\xi I)^{-1} \|\,, \\
 &\leq &\frac{2}{c^2} | \lambda_{\max}(Y^n) - \xi |\,,
\end{eqnarray*} 

and the second statement follows from the convergence of 
$\lambda_{\max}(Y^n)$ to $\xi$ in probability. 
\end{proof}

Next we turn our attention to $\langle (X^n-\xi I)^{-1}u^n, u^n \rangle$ on 
the event where $(X^n-\xi I)$ is invertible. 

\begin{lem}
\label{lm:inp}
Let $\mathcal E_n^{\ref{lm:inp}}$ be the event where $(X^n-\xi I)$ is invertible. Then, 
$1_{\mathcal E_n^{\ref{lm:inp}}} \to_n 1$ in probability, and 
    \begin{align*}
    1_{\mathcal E_n^{\ref{lm:inp}}}  \frac{1}{\| u^n\|^2} 
 \left\langle (X^n-\xi I)^{-1}u^n, u^n \right\rangle 
 \quad \xrightarrow[n\to \infty]{\mathbb{P}} \quad -\frac{1}{\xi}\,.  
    \end{align*}
\end{lem}
\begin{proof}
The convergence $1_{\mathcal E_n^{\ref{lm:inp}}} \to_n 1$ in probability follows
obviously from, \emph{e.g.}, Theorem~\ref{rho-sparse}. Arguing as in the proof
of Lemma~\ref{benaych}, we furthermore have 
\[
1_{\mathcal E_n^{\ref{lm:inp}}} \det(I-\xi^{-1}(X^n+u^n(u^n)^\star)) = 
 1_{\mathcal E_n^{\ref{lm:inp}}} 
 \left( 1+ \left\langle u^n, (X^n-\xi I)^{-1}u^n \right\rangle\right) 
   \det (I - \xi^{-1} X^n)\, . 
\]
By Assumptions~\ref{ass:E} and \ref{ass:vectors}, the sequence $(\| u^n \|)$
converges to a limit $\beta > 0$ along a subsequence that we still denote as
$(n)$. We fix this subsequence. Setting 
$$
\check{E}^n=u^n (u^n)^\star\quad \textrm{and}\quad \check{Y}^n = X^n + \check{E}^n\ ,
$$ and defining 
the $\HH^2$--valued random vector $[ q^\cY_n \ q^X_n ]^T$ as 
\[
\begin{pmatrix} q_n^\cY(z) \\ q_n^X(z) \end{pmatrix} = 
 \begin{pmatrix} \det(I- z \cY^n) \\ \det(I - z X^n) \end{pmatrix} \ , 
\]
we easily see that the sequence $([ q^\cY_n \ q^X_n ]^T)$ is tight in the 
space $\HH^2$ equipped with the product distance, and furthermore, 
by inspecting again the proof of Theorem~\ref{characteristic_poly_thm}
(in particular, Proposition~\ref{tr(Y)-tr(X)-det} with $\check{E}^n = u^n (u^n)^\star$ 
and Lemma~\ref{lemma_moments_of_X^n}), that 
\[
\begin{pmatrix} q_n^\cY \\ q_n^X \end{pmatrix} 
\ \xrightarrow[n\to\infty]{\text{law}} \
\kappa \exp(-F) \begin{pmatrix} b_\infty \\ 1 \end{pmatrix} \qquad \textrm{with}\quad b_\infty(z) = 1 - \beta^2 z\, .
\]
By Slutsky's theorem, we then get that 
\[
1_{\mathcal E_n^{\ref{lm:inp}}} \begin{pmatrix} q_n^\cY \\ q_n^X \end{pmatrix} 
\quad \xrightarrow[n\to\infty]{\text{law}} \quad 
\kappa \exp(-F) \begin{pmatrix} b_\infty \\ 1 \end{pmatrix} . 
\]
By Skorokhod's representation theorem, there exists a sequence of
$\C^2$--valued random variables $([ p^\cY_n \ p^X_n ]^T)$ and a $\C^2$--valued
random variable $([ p^\cY_\infty \ p^X_\infty ]^T)$ on a probability space
$\widetilde{\Omega}$, such that 
$$\begin{pmatrix} p^\cY_n \\ 
p^X_n\end{pmatrix} \ \stackrel{\text{law}}{=}\ 
\1_{\mathcal E_n^{\ref{lm:inp}}} \begin{pmatrix} q^\cY_n(1/\xi) \\ q^X_n(1/\xi) \end{pmatrix}\qquad \textrm{and}\qquad \begin{pmatrix}p^\cY_\infty \\
p^X_\infty \end{pmatrix}\  \stackrel{\text{law}}{=}\ \kappa(1/\xi) \exp(-F(1/\xi)) \begin{pmatrix}
b_\infty(1/\xi) \\ 1 \end{pmatrix}\ ,
$$ and $([ p^Y_n \ p^X_n ]^T)$ converges to $([
p^\cY_\infty \ p^X_\infty ]^T)$ for all $\tilde \omega \in \widetilde{\Omega}$.
Recalling that $\kappa \exp(-F) \neq 0$, it holds that the random variable
$p^\cY_n / p^X_n$ converges pointwise to $p^\cY_\infty / p^X_\infty
\stackrel{\text{law}}{=} b_\infty(1/\xi)$. This implies that
\[ 
1_{\mathcal E_n^{\ref{lm:inp}}} 
 \left( 1+ \left\langle u^n, (X^n-\xi I)^{-1}u^n \right\rangle\right) 
\ =\ 1_{\mathcal E_n^{\ref{lm:inp}}} \frac{q^\cY_n(1/\xi)}{q^X_n(1/\xi)} 
 \quad \xrightarrow[n\to\infty]{\mathbb P} \quad b_\infty(1/\xi)\ =\ 1 - \frac{\beta^2}{\xi}\,, 
\]
and the result of the lemma follows. 
\end{proof}

It remains to establish an asymptotic result for $\frac{1}{\| u^n\|^2}\|
(X^n-\xi I)^{-1}u^n\|^2$. It will be more convenient to work with the
hermitisation $H_n(\xi)$ of $X^n$ defined in \eqref{def:hermit}. Furthermore, it will also be convenient to introduce a small parameter $\eta>0$ and work on the resolvent $(H^n-zI)^{-1}$ of $H^n$ evaluated at $z=i\eta$. Specifically:

\begin{lem}
\label{H-1} 
There exists a sequence of events $(\mathcal E_n^{\ref{H-1}})$ such that 
$H^n$ is invertible on $\mathcal E_n^{\ref{H-1}}$, 
$\mathcal E_n^{\ref{l-mu}} \subset \mathcal E_n^{\ref{H-1}}$, and 
\[
\1_{\mathcal E_n^{\ref{H-1}}} \|(H^n)^{-1}-(H^n - i\eta I )^{-1}\| 
  \quad \leq\quad C_{\ref{H-1}} \eta 
\] 
for some constant $C_{\ref{H-1}} > 0$. 
\end{lem}

\begin{proof}
Recall that $\lambda$ is an eigenvalue of the Hermitian matrix $H^n$
if and only if $\lambda$ or $-\lambda$ is a singular value of $X^n - \xi I$. 
Thus, the event 
\[
\mathcal E_n^{\ref{H-1}} = \left\{s_n(X^n-\xi I)\geq c\right\}
\]
where $c > 0$ is the one chosen in the proof of Lemma~\ref{l-mu} satisfies the first two
assertions of the statement. 

On the event $\mathcal E_n^{\ref{H-1}}$, it holds that $\|(H^n)^{-1}\| \leq
1/c$. On the same event, since the singular values of $(H^n - i \eta I)^{-1}$ are
of the form $1 / |\lambda_k - i\eta|$ where the $\lambda_k$'s are the real
eigenvalues of $H^n$, we obtain that $\|(H^n - i\eta I)^{-1} \| \leq 1/c$. 
By the resolvent identity, on this event, we therefore obtain the following estimate: 
\[
\| (H^n)^{-1} - (H^n - i \eta I)^{-1} \| \quad =\quad  
\| (H^n)^{-1} (H^n - i \eta I)^{-1} \eta \| \quad \leq\quad  
\| (H^n)^{-1} \| \times \| (H^n - i \eta I)^{-1} \|\, \eta  \quad \leq \quad
\frac{\eta}{c^2}\,. 
\]
\end{proof}
For the resolvent $(H^n-i\eta I)^{-1}$ we have that
\begin{lem}
\label{H->m} 
    Consider a sequence of deterministic unit vectors $\tilde w^{2n}\in \C^{2n}$ satisfying 
    $$
    \tilde w_i^{2n}=0\quad \textrm{for}\quad i\in \{n+1,\cdots, 2n\}\,,
    $$ 
    then the following limit holds:
    \begin{align*}
       \|(H^n-i\eta I)^{-1}\tilde w^{2n}\|^2 \quad \xrightarrow[n\to\infty]{\mathbb{P}} \quad  \frac{\Im (m^{\xi}(i \eta))}{\eta}\ ,
    \end{align*}
    where $m^{\xi}$ is the Stieltjes transform of the probability measure $\xi^{\xi}$ defined in the statement of Theorem \ref{thm_gaussians_properties}.
\end{lem}
\begin{proof}
Denoting as $\Im M = (M - M^\star)/(2 i)$ the imaginary part of a complex matrix, 
it holds by the resolvent's identity that 
\[
 \left((H^n-i\eta I)^{-1}\right)^\star (H^n-i\eta I)^{-1}=
 \frac{1}{\eta} \Im \left((H^n-i\eta I)^{-1}\right).
\]
From this, we conclude that 
\begin{align*}
    &\|(H^n-i\eta I)^{-1}\tilde w^{2n}\|^2=\left\langle (H^n-i\eta I)^{-1}\tilde w^{2n},(H^n-i\eta I)^{-1}\tilde w^{2n}\right\rangle= \left\langle \left((H^n-i\eta I)^{-1}\right)^\star (H^n-i\eta I)^{-1} \tilde w^{2n}, \tilde w^{2n} \right\rangle \\&=\left\langle\frac{1}{\eta} \Im \left((H^n-i\eta I)^{-1}\right)w^{2n},w^{2n} \right\rangle , 
\end{align*}
and the claim follows by combining Corollary \ref{bilinear_comparison_cor} 
with Theorem \ref{thm_gaussians_properties}-\ref{thm_bilinear_gaussians}.
\end{proof}
We are now ready to examine the asymptotic behavior of $\frac{1}{\| u^n\|^2}\| (X^n-\xi I)^{-1}u^n\|^2$.
\begin{lem}
\label{normXu} 
    Let $u^n\in \mathbb{C}^n$ be a deterministic vector, then the following limit holds:
    \begin{align*}
    \1_{\mathcal E_n^{\ref{H-1}}}  \frac{\| (X^n-\xi I)^{-1}u^n\|}{\|u^n\|} 
    \quad \xrightarrow[n\to\infty]{\mathbb{P}} 
    \quad   \frac{1}{\sqrt{|\xi|^2-1}}\, .
    \end{align*}
\end{lem} 
\begin{proof}
Denote by $q^{2n}\in \C^{2n}$ the deterministic unit vector defined by 
$$
q^{2n} =\begin{pmatrix}
    u^n/\|u^n\|\\ 0_n
\end{pmatrix}\, .
$$
Recall that $H^n$ is invertible on $\mathcal
E_n^{\ref{H-1}}$ and notice that on this event $(H^n)^{-1}$ writes 
\[
(H^n)^{-1} = \begin{pmatrix}  0 & (X^n - \xi I)^{-\star} \\ (X^n-\xi I)^{-1} & 0 
 \end{pmatrix} \,.
\]
In particular 
\[
 \1_{\mathcal E_n^{\ref{H-1}}}  \frac{\|(X^n-\xi I)^{-1}u^n\|}{\| u^n \|} = 
 \1_{\mathcal E_n^{\ref{H-1}}}  \|(H^n)^{-1} q^{2n}\|\, .
\]
Fix an arbitrarily small $\varepsilon > 0$ and choose $\eta > 0$ small enough  
so that 
\[
C_{\ref{H-1}} \eta \,\leq\, \frac{\varepsilon}{2} \qquad \textrm{and}\qquad \sqrt{\frac{\Im m^\xi(\eta)}{\eta}} > 
  \frac{1}{\sqrt{|\xi|^2-1}} - \frac{\varepsilon}{4}\,, 
\]
which is possible by Corollary \ref{cor_for_im(eta/eta}. With this choice, we have by Lemma~\ref{H-1} 
$$
\left|\, \1_{\mathcal E_n^{\ref{H-1}}}  
  \left\| (H^n)^{-1} q^{2n} \right\| 
  - \1_{\mathcal E_n^{\ref{H-1}}}  \left\| (H^n-i\eta I)^{-1} q^{2n} \right\| \,\right| \quad \leq  \quad 
 \1_{\mathcal E_n^{\ref{H-1}}}  
  \left\| (H^n)^{-1} q^{2n} - 
(H^n-i\eta I)^{-1} q^{2n} \right\| \quad \leq\quad  \frac{\varepsilon}{2} \,. 
$$
Now
\begin{eqnarray*}
\left\{ \left| \1_{\mathcal E_n^{\ref{H-1}}}  \frac{\| (X^n-\xi I)^{-1}u^n \|}
  {\| u^n \|} - \frac{1}{\sqrt{|\xi|^2-1}} \right| \geq \varepsilon \right\} &\subset &
 \left\{ \left| 
 \1_{\mathcal E_n^{\ref{H-1}}}  
  \left\|(H^n - i\eta I)^{-1} q^{2n} \right\| 
- \frac{1}{\sqrt{|\xi|^2-1}} \right| \geq \frac{\varepsilon}{2} \right\} \\
&\subset& \left\{ 
   \left| \left\| (H^n - i\eta I)^{-1} q^{2n} \right\| 
  - \sqrt{\frac{\Im m^\xi(\eta)}{\eta}} \right| \geq \frac{\varepsilon}{4} 
   \right\} \, .
\end{eqnarray*}
Taking the probability of both events, combined with Lemma~\ref{H->m}, yields the desired result.
\end{proof}
We conclude with the proof of Theorem \ref{thm_eigenvector}.
\begin{proof}[Proof of Theorem \ref{thm_eigenvector}]
We need to show that 
\[
 \left| \left\langle\frac{u^n}{\|u^n\|},\tilde u^n \right\rangle\right|^2 \quad 
\xrightarrow[n\to\infty]{\mathbb{P}} \quad 1 - \frac{1}{|\xi|^2}. 
\]
To this end, we are allowed to multiply the left hand side with 
$\1_{|\sigma^+_\varepsilon(Y^n)| = 1} \1_{\mathcal E_n^{\ref{l-mu}}}$ 
which converges to one in probability by Corollary~\ref{lmax-uv} 
and Lemma~\ref{l-mu}. 

On the event $\left\{ |\sigma^+_\varepsilon(Y^n)| = 1 \right\}$, the right eigenspace of 
$Y^n$ associated with $\lambda_{\max}(Y^n)$ is one-dimensional. By 
Lemma~\ref{benaych}, we are therefore reduced to showing that 
\[
\1_{|\sigma^+_\varepsilon(Y^n)| = 1} \1_{\mathcal E_n^{\ref{l-mu}}} 
 \frac{\left| \left\langle u^n, (X^n - \lambda_{\max}(Y^n) I)^{-1} u^n 
   \right\rangle\right|^2} 
 {\| u^n \|^2 \| (X^n - \lambda_{\max}(Y^n) I)^{-1} u^n \|^2} \quad 
\xrightarrow[n\to\infty]{\mathbb{P}}\quad  1 - \frac{1}{|\xi|^2}\,. 
\]
Noticing that 
$\mathcal E_n^{\ref{l-mu}} \subset \mathcal E_n^{\ref{lm:inp}}$, we obtain
by Lemmas~\ref{l-mu} and~\ref{lm:inp} that 
\[
    \1_{\mathcal E_n^{\ref{l-mu}}}  \frac{1}{\| u^n\|^2} 
 \left\langle (X^n-\lambda_{\max}(Y^n) I)^{-1}u^n, u^n \right\rangle 
 \quad \xrightarrow[n\to\infty]{\mathbb{P}} \quad
 -\frac{1}{\xi}\,.  
\]
By Lemmas~\ref{l-mu}, \ref{H-1} and~\ref{normXu}, it holds that 
\[
 \1_{\mathcal E_n^{\ref{l-mu}}}  
\frac{\| (X^n-\lambda_{\max}(Y^n) I)^{-1}u^n\|}{\|u^n\|} 
 \quad \xrightarrow[n\to\infty]{\mathbb{P}} \quad\frac{1}{\sqrt{|\xi|^2-1}}\,, 
\]
and the result is obtained through a direct calculation. 
\end{proof}

\section{Open problems}

We now present several open problems that emerge naturally from our results. Most of these appear approachable using refinements of existing techniques, while one in particular—concerning assumptions on Theorem \ref{thm_eigenvector}—poses a more significant theoretical challenge and remains largely unresolved.

\begin{problem}[sparser regimes]
The bounds in \eqref{bound_on_var} and \eqref{bound_on_mean} tend to zero even when $K_n$ remains bounded as $n \to \infty$. Our current methods already yield an analogue of \eqref{thm_statement_char_pol} in the case $K_n = K > 0$. To fully extend the result, one must compute the moments of $\operatorname{Tr}(X^n)$, as in Lemma~\ref{lemma_moments_of_X^n}. The limiting distribution is not Gaussian—in the directed Erdős–Rényi case, for instance, the non-Gaussian limit is derived in \cite{coste2023sparse}.
\end{problem}

\begin{problem}[types of sparsity]
Extend the analysis to alternative sparsity regimes beyond that defined in \eqref{sparser_matrices}. For example, consider the Hadamard product of an i.i.d. matrix with the adjacency matrix of a $K_n$-regular graph, uniformly sampled from the space of such graphs. The interplay between randomness and structured sparsity presents new analytical challenges.
\end{problem}

\begin{problem}[unbounded eigenvalues of \texorpdfstring{$E^n$}{En}]
Proposition~\ref{tr(Y)-tr(X)-det} remains valid if $\|E^n\| = O(n^{o(1)})$. Investigate whether, after proper normalization, the sequence $q_n(z)$ remains tight and whether Theorem~\ref{characteristic_poly_thm} continues to hold when $\|E^n\| \to \infty$ as $n \to \infty$.
\end{problem}

\begin{problem}[assumptions on Theorem \ref{thm_eigenvector}]
Can one remove the distributional and sparsity assumptions in Theorem~\ref{thm_eigenvector}? Doing so would require asymptotic lower bounds on the least singular value $s_n(X^n - \xi I)$. Our approach depends on the universality results of \cite{brailovskaya2024universality}, which justify these extra assumptions. Removing them appears to be a substantially harder problem and is currently out of reach.
\end{problem}

\begin{problem}[eigenvectors of finite-rank perturbation]
Generalize Theorem~\ref{thm_eigenvector} to the case of a deformation with an arbitrary finite
rank, similarly to what was done in the Hermitian case by, \emph{e.g.}, \cite{benaych2011eigenvalues}). This generalization is useful for many applicative contexts
where the matrix $E^n$ bears an ``information'' buried in the sparse noise matrix $X^n$. 
\end{problem}

\section*{Funding}

M.\ Louvaris has received funding from the European Union’s Horizon~2020 research
and innovation programme under the Marie Skłodowska--Curie grant agreement
No.~101034255.

\section*{Acknowledgement}
The authors wish to thank Ada Altieri for fruitful discussions.

\bibliographystyle{alpha}

\bibliography{references}
\addcontentsline{toc}{chapter}{bibliography}
\end{document}